\documentclass[12pt]{amsart}
\usepackage{amsmath,amsthm,amssymb}
\usepackage{tikz,pgf}
\usepackage{standalone}
\usepackage[citestyle=alphabetic,bibstyle=alphabetic]{biblatex}
\usepackage{a4wide}
\usepackage{float}
\usepackage{dsfont}
\usepackage{graphicx}
\usepackage{svg}
\usepackage{caption}

\newcommand{\bb}[1]{\mathbb{#1}}
\newcommand{\R}{\mathbb{R}}

\newcommand{\N}{\mathbb{N}}

\newcommand{\C}{\mathbb{C}}

\newcommand{\1}{\mathds{1}}
\newcommand{\Ai}{\text{Ai}}
\newcommand{\cov}[0]{\text{Cov}}
\newcommand{\var}[0]{\text{Var}}

\newcommand{\argmax}{\text{arg}\,\text{max}}
\newtheorem{theorem}{Theorem}
\newtheorem{proposition}{Proposition}[section]
\newtheorem{corollary}[proposition]{Corollary}
\newtheorem{lemma}[proposition]{Lemma}
\numberwithin{equation}{section}

\theoremstyle{definition}
\newtheorem*{definition}{Definition}

\theoremstyle{remark}

\theoremstyle{definition}

\theoremstyle{definition}
\newtheorem{remark}{Remark}[section]

\title{On the Hard Edge Limit of the Zero Temperature Laguerre $\beta$-Corners Process}
\author{Matthew Lerner-Brecher}

\begin{document}
\maketitle
\begin{abstract}
We study the hard edge limit of a multilevel extension of the Laguerre $\beta$-ensemble at zero temperature. In particular, we show that asymptotically the ensemble is given by Gaussians with covariance matrix expressible in terms of the Fourier-Bessel series. These Gaussians also have an explicit representation as the partition functions of additive polymers arising from a random walk on roots of the Bessel functions. Our approach builds off of the one introduced in \cite{GK} and is rooted in using the theory of dual and associated polynomials to diagonalize transition matrices relating levels of the ensemble.
\end{abstract}

\tableofcontents

\section{Introduction}

\subsection{Background}
The general $\beta$-ensemble is a distribution on decreasing tuples $(x_1,\ldots, x_n)$ of $n$ reals with density of the form
\begin{equation} \label{eq:general_beta_dist}
\frac{1}{Z} \prod_{i=1}^n w(x_i) \prod_{1 \le i < j \le n} |x_i - x_j|^{\beta},
\end{equation}
where $Z$ is a normalizing constant and $w$ is some positive-valued weight function. Such ensembles have a rich history of study on account of their connections to both statistical mechanics and random matrix theory. In the former, \eqref{eq:general_beta_dist} describes both the distribution of particles in a Coulomb log-gas system and the squared ground state wave function of the Calogero-Sutherland model (see \cite{Ox} for a primer). Correspondingly, $\beta$ is known as the \textit{inverse temperature} of the model. \par 
In random matrix theory, the impetus for studying these distributions comes from three special cases referred to as the \textit{classical ensembles}, which are defined by
\[ 
w(x) = \begin{cases} \exp\left(-\beta x^2/4\right) & \text{ Hermite/Gaussian Ensemble} \\
x^{\beta(\alpha+1)/2-1}\exp(-\beta x/2) & \text{ Laguerre Ensemble} \\
x^{\beta(p+1)/2-1}(1-x)^{\beta(q+1)/2-1} & \text{ Jacobi Ensemble}
\end{cases} 
\]
Here $\alpha, p, q > -1$, and the respective supports of $w(x)$ are $\R, \R_{>0}, (0,1)$. Each of the classical ensembles can be realized as the distribution of the eigenvalues of a random tridiagonal matrix \cite{ED, KN}. This tridiagonal approach is a more recent innovation however, and historically, the interest in these three ensembles comes from alternative matrix distributions in the cases $\beta = 1,2,4$. To describe these, let $X$ be an $(n+\alpha) \times  n$ matrix, whose entries are independent standard normal Gaussians and either real, complex, or quaternionic corresponding to $\beta = 1,2,4$ respectively. The distribution of the squared singular values of $X$ has density given by the Laguerre ensemble, and when $\alpha = 0$, the distribution of the eigenvalues of the Hermitian matrix $\frac{1}{2}(X+X^*)$ has density given by the Hermite ensemble. The realization of the Jacobi ensemble is less relevant to our discussion, but can be found in \cite{For10} for the interested reader. 
\par 
Motivated by these matrix distributions at $\beta = 1,2,4$, one can extend the single level general $\beta$-ensemble to a multilevel ensemble. Indeed, if one sets $\lambda^k = (\lambda_{1,k},\ldots, \lambda_{k,k})$ to be the decreasing eigenvalues of the top left $k \times k$ corner of $\frac{1}{2}(X+X^*)$ (resp. $X^*X$), then for each $k$, $\lambda^k$ will be an instance of the Hermite (resp. Laguerre) ensemble. The joint distribution of $\lambda^1,\cdots, \lambda^n$ will be referred to as a \textit{$\beta$-corners process} (or $\beta$-minors process) and each individual $\lambda^i$ as a level of the ensemble. These multilevel distributions have been the attention of much recent study due to their connection to tiling models and symmetric polynomials. For the former, consider tiling a regular hexagon by lozenges (rhombuses) whose sides have equal length and are parallel to those of the hexagon. If one selects a tiling of this hexagon uniformly at random, then with probability approaching one as the length of the lozenges goes to zero, the tiling will be separated into different regions by a deterministic curve. In a ``frozen'' region, only one type of lozenge will occur, whereas in a ``liquid'' region all types occur with varying densities. The Gaussian corners process with $\beta = 2$ describes the asymptotic behavior of the tiling at a point of tangency between the boundary of the domain and the deterministic curve. See \cite{OR, JN} for the original appearance of GUE corners in tiling models and \cite{AG} for figures and a generalization of this result to domains other than a regular hexagon.
\par 
On the algebraic side, corners processes have a close relationship to Macdonald polynomials and their degenerations to the Heckman-Opdam and multivariate Bessel functions \cite{BC14}. Notably, these ideas were exploited in \cite{GM} to obtain expected characteristic polynomials for sums and products of random matrices and in \cite{BG15} to prove that the global fluctuations of the multilevel Jacobi ensemble are described by the Gaussian Free Field. The density of the Laguerre $\beta$-corners process will be presented in Proposition \ref{prop:laguerre_infinity_corners}; for the form of the density in the Hermite, Jacobi, and general cases see the aforementioned papers \cite{GM, BG15} and for the original study on the joint distribution of matrix corners see \cite{Ner}. All of these distributions can be naturally interpolated to $\beta$ other than 1,2,4 allowing one to define multilevel $\beta$-ensembles in general.
\par 
Unfortunately, for general $\beta$ these multilevel ensembles are more difficult to study, in part because tridiagonal approach of \cite{ED} is no longer applicable. However, recent work indicates that there may be another special $\beta$ value as amenable to study and relevant to the general theory as $\beta = 1,2,4$: the zero temperature limit $\beta\to\infty$. In addition to having significance from a statistical mechanics standpoint, this limit exhibits crystallization to a deterministic lattice, Gaussian fluctuations about said lattice \cite{GM}, and single level covariance matrix diagonalizable by orthogonal polynomials \cite{V, AV}. Additionally, \cite{ED04} showed that one can well-approximate the Hermite and Laguerre $\beta$-ensembles by Taylor expanding about the zero temperature limit (see \cite{GK} for relevant figures).
\par 
The focus of this paper will be on the zero temperature limit of these multilevel ensembles, particularly in the Laguerre case. However, before we restrict our attention to this case, there is another interesting regime we feel we should mention: the high temperature limit. The straightforward limit where $n$ is fixed and $\beta \to 0$ is less novel as the interactions between the eigenvalues disappear. Instead, attention in the literature has been focused on the case where $n\to\infty$ and $\beta n$ converges to a constant $\gamma$. In this limit, one obtains multiple law of large number results differing from the standard ones for $\beta n \to \infty$ \cite{BCG, TT}. Recent results also suggest that this high temperature limit is closely connected to the zero temperature one particularly for the Gaussian ensemble. In particular, the limiting empirical measure for the Gaussian ensemble at high temperature equals the orthogonality measure of the associated Hermite polynomials, which will play an important role for us in section \ref{sec:Q_m_properties} \cite{BCG}. Additionally, in the two regimes, one can identify parameters arising from various convolutions associated to matrix addition \cite{BCG, Xu}, and a recent paper by Forrester \cite{For22b} found an equivalence between linear statistics. See discussions in \cite{BCG, Xu} for further exploration of this topic.

\subsection{The Main Results}
When studying the extremal eigenvalues of the classical ensembles at $\beta = 1,2,4$, different behavior emerges depending on whether the eigenvalues are constrained by an endpoint of the support, as is the case at 0 for the Laguerre ensemble and 0,1 for the Jacobi ensemble, or whether they remain unbounded. The former case is referred to as a hard edge, and under proper scaling, the $n \to \infty$ limit of the correlation kernel of the eigenvalues nearest to the endpoint can be evaluated in terms of the Bessel functions. Meanwhile in the latter case, referred to as a soft edge, a similar limit yields an analogous expression involving the Airy function \cite{For10}. In addition to being of general interest in random matrix theory, see e.g. \cite{Ram, CMS13}, edge limits play an important role in quantum chromodynamics, see e.g. \cite{SV93, DHCR}; statistical mechanics, see e.g. \cite{LDS, Seo}; and numerical analysis, see e.g. \cite{EGP}.
\par 
In this paper we follow the lead of two groups of research: \cite{GK}, which studied zero temperature edge limits of the Gaussian $\beta$-corners process, and \cite{AHV,And21}, which together analyzed similar single level limits for several of the classical ensembles. In particular, this paper will study the hard edge limit of the zero temperature Laguerre $\beta$-corners processes. As the aforementioned papers, particularly the former, are vital to our methods, we will give a more in-depth discussion of them but forestall it until the next section.
\par 
To describe our results, let $X_1,X_2,\ldots,X_n$ be a sequence of matrices whose entries are independent standard normal Gaussians. As before these entries can be real, complex, or quanternionic corresponding to $\beta = 1,2,4$ respectively. We also assume $X_k$ is $N \times k$ and forms the leftmost block of $X_{k+1}$. Let $\lambda^k = (\lambda_{1,k}, \cdots, \lambda_{N \wedge k,k})$ be the nonzero squared singular values of $X_k$ listed in decreasing order where $N \wedge k = \min(N,k)$. We will refer to the joint distribution of $\lambda^1,\ldots, \lambda^n$ as the \textit{Laguerre $\beta$-corners process with $n$ rows centered at $N$}. In Section \ref{sec:zero_temp}, it will be shown that for $\beta = 1,2,4$ this process has density proportional to:
\begin{equation} \label{eq:intro_laguerre_corners_density}
\prod_{i=1}^N \lambda_{i,N}^{\frac{\beta}2-1}e^{-\frac{\beta\lambda_{i,n}}2}\prod_{1 \le i < j \le N} (\lambda_{i,n} - \lambda_{j,n}) \prod_{k=1}^{n-1}\prod_{1 \le i < j \le N \wedge k} (\lambda_{i,k} - \lambda_{j,k})^{2-\beta} \prod_{k=1}^{n-1}\prod_{a = 1}^{k \wedge N}\prod_{b=1}^{k+1 \wedge N} |\lambda_{a,k} - \lambda_{b,k+1}|^{\frac{\beta}2 - 1}.
\end{equation}
Hence, the ensemble can be naturally defined for all $\beta > 0$. Note for $k < N$, $X_k^*X_k$ forms the top $k \times k$ corner of $X_{k+1}^*X_{k+1}$ so this is indeed an extension of the corners procedure of the previous section. 
\par 
In order to state our main theorem we need a few final definitions. Let the generalized Laguerre polynomial $L^{(\alpha)}_n(x)$ be the polynomial solution to
\begin{equation} \label{eq:laguerre_diffeq}
xy'' + (\alpha + 1 - x)y' + ny = 0,
\end{equation}
with leading coefficient $\frac{(-1)^n}{n!}$.\footnote{For $\alpha \ge 0$, $\{L^{(\alpha)}_i(x)\}_{i \in \N}$ equivalently are the orthogonal polynomials with respect to the weight $x^{\alpha}e^{-x}$ on $\R_{+}$.} Additionally, define $J_s(x)$ to be the Bessel function of the first kind of order $s$; let $j_{1,s} < j_{2,s} < \cdots$ be the real roots of $J_s(x)$ for $s \ge 0$; and set $j_{i-s, s} = j_{i,-s}$ for $s < 0$. These objects are relevant to our theory as the singular value $\lambda_{k+1-i,k}$ tends to the $i$-th smallest root of the Laguerre polynomial $L^{(N-k)}_k(x)$ as $\beta\to\infty$. When $k-N$ and $i$ equal fixed constants, as $N\to\infty$, this root is approximately $\frac{j_{i,N-k}^2}{4N}$. Indeed, the reason for the offset in defining $j_{i-s, s}$ for $s<0$ is so that this approximation holds even when $N-k < 0$. We are now ready to state our main result:
\begin{theorem} \label{thm:main_hard_edge} For given positive integers $n \ge N$, let reals $\lambda_{1,k} > \lambda_{2,k} > \cdots > \lambda_{N\wedge k, k}$ be distributed so that their joint density\footnote{With respect to the Lebesgue measure} across $1 \le k \le n$ is proportional to \eqref{eq:intro_laguerre_corners_density}. Additionally, let $l_{i,k}$ denote the $i$-the largest root of the generalized Laguerre polynomial $L^{(N-k)}_k(x)$ for all $k \ge 1$. There exists a Gaussian process $\{\zeta_{a,s} : a \ge \max(-s, 1)\}$ indexed by integers $a,s$ such that in the sense of convergence of finite dimensional distributions one has
\[\lim_{N\to\infty}\lim_{\beta \to\infty} N\sqrt{\beta} (\lambda_{N+1-s-a,N-s} - l_{N+1-s-a, N-s}) = \zeta_{a,s} \qquad \qquad a \ge \max(-s, 1).\]
Here we assume $n\to\infty$ with $N$ in such a way that $(n - N) \to\infty$.\footnote{The projection of \eqref{eq:intro_laguerre_corners_density} onto $n' < n$ rows will have density given by \eqref{eq:intro_laguerre_corners_density} with $n$ replaced by $n'$. Thus the specific choice of $n \ge N$ does not affect the limit so long as all variables are well-defined.} Furthermore, this Gaussian process has explicit covariances given by 
\begin{equation} \label{eq:covariance_limit_hard_edge}
\cov(\zeta_{a,s}, \zeta_{b,t}) = \frac{j_{a,s}j_{b,t}}{2}\int_0^1 \frac{J_{s}(j_{a,s}\sqrt{1-y})}{J_s'(j_{a,s})} \frac{J_{t}(j_{b,t}\sqrt{1-y})}{J_t'(j_{b,t})} \frac{(1-y)^{|s-t|/2}}{y}dy,
\end{equation}
for all $a,s,b,t$ such that $\zeta_{a,s}, \zeta_{b,t}$ are well-defined.
\end{theorem}
The proof of Theorem \ref{thm:main_hard_edge} will be given in Section \ref{sec:proof_main_hard_edge}. A single level version of this result can be found in \cite{And21}.\footnote{Any difference between the statement of our result and theirs is accounted for by the fact that they study the singular values of $X_k$ instead of the squared singular values.} In Proposition \ref{prop:laguerre_infinity_corners}, we will show that if one only takes the $\beta \to \infty$ limit, then the fluctuations still converge in the same sense to a multivariate Gaussian. The resulting process will be referred to as the \textit{Laguerre $\infty$-corners process} and the limit of $\sqrt{\beta}(\lambda_{i,k} - l_{i,k})$ will be given by the variable $\xi_{i,k}$. One analogously defines the Gaussian $\infty$-corners for the Gaussian $\beta$-corners process \cite{GK}.
\par 
As an addendum to our main result, we will show that this Gaussian process can be constructed in another fashion. In particular, the elements of this process equal the partition functions of additive polymers arising from a random walk on the zeros of the Bessel function. Varying the zero at which this random walk starts changes which element of the process one gets. To be more precise, let $j_{i,v} = 0$ for $i < -v$ and $S_v = \{j_{i,v}\}_{i=1}^{\infty}$ where $v$ is some integer. For any integer $\alpha$, define the \textit{Bessel-$\infty$ random walk started at level $\alpha$} to be a Markov chain with state space at time $t$ given by $S_{\alpha - t}$ and transition probability from $j_{a,v}\neq 0$ to $j_{b,v-1}$ given by
\begin{equation} \label{eq:transition_provs_Bessel_infty} P^{v,v-1}(a \to b) = \frac{4j_{a,v}^2}{(j_{a,v}^2 - j_{b,v-1}^2)^2}.
\end{equation}
In the case that $j_{a,v} = 0$, the random walk instead transitions to one of the 0 values in $S_{v-1}$ with uniform probability. For $v_2 < v_1$, we also define $P^{v_1, v_2}$ to be given by composing the transition matrices $P^{v_1,v_1-1}, \cdots, P^{v_2+1,v_2}$. In Section \ref{sec:additive_polymer_limit} we will show that these transition probabilities add up to 1 and thus the Bessel-$\infty$ random walk is indeed a well-defined random walk.
\par 
To illustrate this, Figure 1 gives two different examples of Bessel-$\infty$ random walks started at level $\alpha = 1$. Here, level $i$ contains the elements of $S_i$ and one can note that for $i < 0$ level $-i$ contains $i$ copies of zero represented by unlabelled points. The first walk, given by the solid line, starts at the smallest root $j_{1,1}$ of $J_1(x)$ and eventually transitions to a zero value at level -2. From there on out, the random walk can only transition to zero values. The second walk, given by the dashed line, starts at the second smallest root of $J_1(x)$ and does not reach a zero value by level -3. With all this terminology we can state our next result as follows:
\begin{theorem} 
\label{thm:secondary_polymer_result}
For integers $v$ and positive integers $a$, let $\eta_{a,v}$ be independent mean zero Gaussians with variance $\frac{j_{a,v}^2}{2}$. The additive polymer partition functions
\[ \zeta_{a,v_1} := \sum_{v=-\infty}^{v_1} \sum_{b=1}^{\infty} P^{v_1, v}(a \to b)\eta_{b,v}\]
are zero mean Gaussians with covariances given by \eqref{eq:covariance_limit_hard_edge}.
\end{theorem}
These random variables $\zeta_{a,v_1}$ can be seen to be analogous to the $\textit{Airy}_{\infty}$\textit{ line ensemble} of \cite{GK}, which arises instead from a continuous time random walk on the roots of the Airy function.
\begin{figure}[h] 
\tikzstyle{particle}=[circle, draw, fill=black, minimum size = 3mm, inner sep=0pt]
\begin{tikzpicture}[xscale=.5, yscale = .6]

%Particles

%\node[particle, fill=red] (j11) at (3.8317,-2) {};
\node[particle] (j11) at (3.8317,-2) {};
%\node[particle, fill = blue] (j21) at (7.0156,-2)  {};
\node[particle] (j21) at (7.0156,-2)  {};

\node[particle] (j31) at (10.1735,-2) {};
\node[particle] (j41) at (13.3237,-2) {};
\node[particle] (j51) at (16.4706,-2) {};

\node[particle] (j10) at (2.4048,0) {};
%\node[particle, fill = red] (j20) at (5.5201,0) {};
\node[particle] (j20) at (5.5201,0) {};
\node[particle] (j30) at (8.6537,0) {};
\node[particle] (j40) at (11.7915,0) {};
%\node[particle, fill = blue] (j50) at (14.9309,0) {};
\node[particle] (j50) at (14.9309,0) {};

\node[particle] (mj11) at (0,2) {};
\node[particle] (mj21) at (3.8317,2) {};
%\node[particle, fill = red] (mj21) at (3.8317,2) {};
\node[particle] (mj31) at (7.0156,2) {};
%\node[particle, fill = blue] (mj41) at (10.1735,2) {};
\node[particle] (mj41) at (10.1735,2) {};
\node[particle] (mj51) at (13.3237,2) {};
\node[particle] (mj61) at (16.4706,2) {};

\node (center2) at (0,4) {};
%\node[particle,fill=red] (mj12) at (center2.west) {};
\node[particle] (mj12) at (center2.west) {};
\node[particle] (mj22) at (center2.east) {};
\node[particle] (mj32) at (5.1356,4) {};
%\node[particle, fill = blue] (mj42) at (8.4172,4) {};
\node[particle] (mj42) at (8.4172,4) {};
\node[particle] (mj52) at (11.6198,4) {};
\node[particle] (mj62) at (14.7960,4) {};
\node[particle] (mj72) at (17.9598,4) {};

\node[particle] (mj23) at (0,6) {};
\node[particle, xshift=-1.5mm] (mj13) at (mj23.west) {};
\node[particle, xshift=1.5mm,fill=red] (mj33) at (mj23.east) {};
\node[particle, xshift=1.5mm] (mj33) at (mj23.east) {};
\node[particle] (mj43) at (6.3802,6) {};
%\node[particle, fill = blue] (mj53) at (9.7610,6) {};
\node[particle] (mj53) at (9.7610,6) {};
\node[particle] (mj63) at (13.0152,6) {};
\node[particle] (mj73) at (16.2235,6) {};
\node[particle] (mj83) at (19.4094,6) {};

\node at (0,8) {\huge $\vdots$};
\node[rotate=-22] at (7.643,8) {\huge $\vdots$};
\node[rotate=-22] at (11.1,8) {\huge $\vdots$};
\node[rotate=-22] at (14.41,8) {\huge $\vdots$};
\node[rotate=-22] at (17.65,8) {\huge $\vdots$};
\node[rotate=-22] at (20.859,8) {\huge $\vdots$};
\node[rotate=-45] at (24.05,7.9) {\huge $\vdots$};
\node[rotate=90] at (22.6,6) {\huge $\vdots$};
\node[rotate=90] at (21.15,4) {\huge $\vdots$};
\node[rotate=90] at (19.66,2) {\huge $\vdots$};
\node[rotate=90] at (18.12,0) {\huge $\vdots$};
\node[rotate=90] at (19.66,-2) {\huge $\vdots$};

%Labels
\node at (-5,6) {Level $-3$};
\node at (-5,4) {Level $-2$};
\node at (-5,2) {Level $-1$};
\node at (-5,0) {Level $0$};
\node at (-5,-2) {Level $1$};

\node[yshift = 2mm] at (j11.north) {$j_{1,1}$};
\node[yshift = 2mm] at (j21.north) {$j_{2,1}$};
\node[yshift = 1mm, xshift = 5mm] at (j31.north) {$j_{3,1}$};
\node[yshift = 2mm] at (j41.north) {$j_{4,1}$};
\node[yshift = 2mm] at (j51.north) {$j_{5,1}$};

\node[yshift = 2mm] at (j10.north) {$j_{1,0}$};
\node[yshift = 2mm, xshift = 3mm] at (j20.north) {$j_{2,0}$};
\node[yshift = 2mm] at (j30.north) {$j_{3,0}$};
\node[yshift = 2mm] at (j40.north) {$j_{4,0}$};
\node[yshift = 2mm,xshift=3mm] at (j50.north) {$j_{5,0}$};

\node[yshift = 2mm] at (mj11.north) {};
\node[yshift = 2mm, xshift = 4mm] at (mj21.north) {$j_{2,-1}$};
\node[yshift = 2mm] at (mj31.north) {$j_{3,-1}$};
\node[yshift = 2mm,xshift=4mm] at (mj41.north) {$j_{4,-1}$};
\node[yshift = 2mm] at (mj51.north) {$j_{5,-1}$};
\node[yshift = 2mm] at (mj61.north) {$j_{6,-1}$};

\node[yshift = 2mm, xshift=3mm] at (mj22.north) {};
\node[yshift = 2mm] at (mj32.north) {$j_{3,-2}$};
\node[yshift=2mm, xshift=5mm] at (mj42.north) {$j_{4,-2}$};
\node[yshift = 2mm] at (mj52.north) {$j_{5,-2}$};
\node[yshift = 2mm] at (mj62.north) {$j_{6,-2}$};
\node[yshift = 2mm] at (mj72.north) {$j_{7,-2}$};

\node[yshift = 2mm] at (mj23.north) {};
\node[yshift = 2mm] at (mj43.north) {$j_{4,-3}$};
\node[yshift= 2mm] at (mj53.north) {$j_{5,-3}$};
\node[yshift = 2mm] at (mj63.north) {$j_{6,-3}$};
\node[yshift = 2mm] at (mj73.north) {$j_{7,-3}$};
\node[yshift = 2mm] at (mj83.north) {$j_{8,-3}$};`

%Arrows
\draw[->] (j11.east) .. controls +(right:5mm) and +(down:5mm) ..  (j20.south);
\draw[->] (j20.north) .. controls +(up:5mm) and +(down:5mm) ..  (mj21.south);
\draw[->] (mj21.north) .. controls +(up:5mm) and +(down:5mm) ..  (mj12.south);
\draw[->] (mj12.north) .. controls +(up:5mm) and +(down:5mm) ..  (mj33.south);

\draw[dashed, ->] (j21.east) .. controls +(right:5mm) and +(down:5mm) ..  (j50.south);
\draw[dashed, ->] (j50.north) .. controls +(up:5mm) and +(down:5mm) ..  (mj41.south);
\draw[dashed, ->] (mj41.north) .. controls +(up:5mm) and +(down:5mm) ..  (mj42.south);
\draw[dashed, ->] (mj42.north) .. controls +(up:5mm) and +(down:5mm) ..  (mj53.south);

\end{tikzpicture}
\captionsetup{justification=raggedright}
\caption{Two examples of Bessel-$\infty$ random walks started at $\alpha = 1$ represented by a solid and dashed line. Level $i$ here contains the elements of $S_i$. The unlabelled points correspond to values of 0.}
\end{figure}
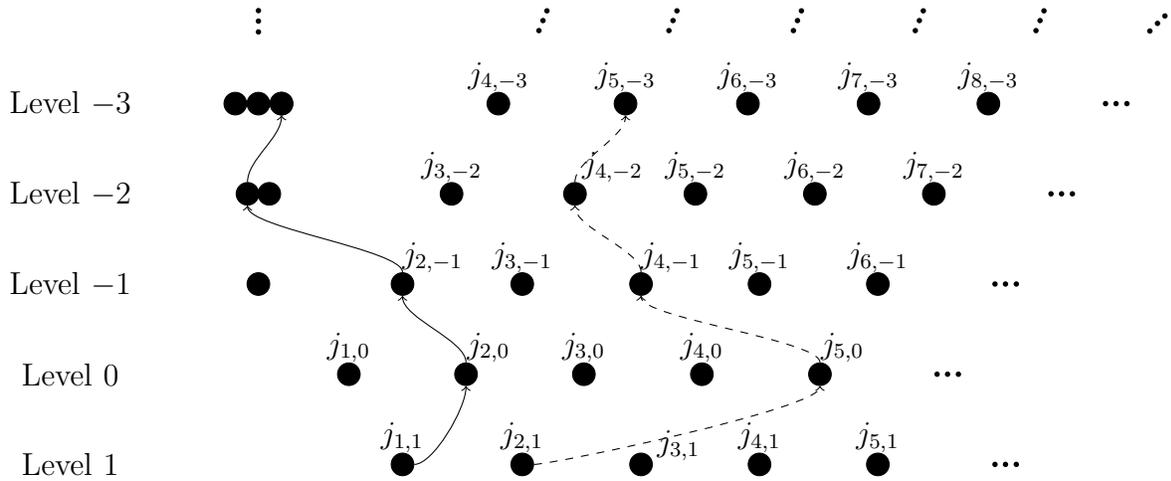

\subsection{Related Literature}
While \cite{GK} and \cite{AHV} used largely different techniques to obtain zero temperature edge limits, a surprising connection between their methods emerged: the presence of dual polynomials. Suppose $\mathcal{P}_0(x),\mathcal{P}_1(x),\ldots$ is a sequence of polynomials satisfying the three-term recurrence
\[x\mathcal{P}_i(x) = \mathcal{P}_{i+1}(x) + b_i\mathcal{P}_i(x) + u_i\mathcal{P}_{i-1}(x),\]
for all $i \ge 1$. Let $M$ be a fixed positive integer. The dual polynomials \cite{BS, VZ} $\mathcal{Q}_0(x),\ldots \mathcal{Q}_{M-1}(x)$ associated to $\mathcal{P}_i(x)$ are the polynomials satisfying $\mathcal{Q}_0(x) = 1, \mathcal{Q}_1(x) = x - b_{M-1}$ and the recurrence
\[x\mathcal{Q}_i(x) = \mathcal{Q}_{i+1}(x) + b_{M-i-1}\mathcal{Q}_i(x) + u_{M-i}\mathcal{Q}_{i-1}(x).\]
\cite{AHV} proved that for appropriate choices of $M$ the dual Hermite/Laguerre/Jacobi polynomials diagonalize the single-level covariance matrices of the corresponding ensemble at zero temperature. In \cite{GK}, these polynomials instead diagonalized certain transition matrices relating consecutive levels of the Gaussian $\infty$-corners process. More precisely, \cite{GK} showed that one can decompose 
\begin{equation} \label{eq:intro_to_transition_matrices}
\xi_{a,k} = \sum_{b=1}^{k+1}\alpha_{a,b}^k\xi_{b,k+1}  + \eta_{a,k},
\end{equation}
where the variables $\eta_{a,k}$ are independent mean zero Gaussians with explicit variance, $\alpha_{a,b}^k$ are the elements of the diagonalizable transition matrix, and $\xi_{a,k}$ are the correlated Gaussians defined the paragraph following Theorem \ref{thm:main_hard_edge}. Through an iteration of this decomposition along with some stochastic results, the authors were able to express $\xi_{a,k}$ in terms of an infinite summation involving only the transition matrices and the independent Gaussians $\eta_{a,k}$. From there, as in \cite{AHV}, the problem reduces to analyzing the dual polynomials in the large $n$ limit. This analysis was conducted in \cite{AHV} by proving that the limiting expression for the dual polynomials satisfies a differential equation, which can be discretely approximated by the above recurrence. \cite{GK} instead obtained contour integration identities for the polynomials, which they used to perform steepest descent asymptotic analysis. 
\par 
Naturally, the results obtained by \cite{GK, AHV} also bear a resemblance to ours. \cite{GK} proved that for the Gaussian $\infty$-corners ensemble, under a different scaling, the edge limit of the covariances from Theorem \ref{thm:main_hard_edge} instead takes the form
\begin{equation} 
\label{eqn:Hermite_Soft_Edge_Covariances}
2\int_0^{\infty} \frac{\Ai(\mathfrak{a}_i+y)\Ai(\mathfrak{a}_j+y)}{\Ai'(\mathfrak{a}_i)\Ai'(\mathfrak{a}_j)}\exp(-|t-s|y)\frac{dy}{y}
\end{equation}
where $\Ai(x)$ is the airy function and $\mathfrak{a}_1 > \mathfrak{a}_2 > \cdots$ are its real roots. The functions $\Ai(\mathfrak{a}_i+y)$ and $J_s(j_{a,s}\sqrt{1-y})$ arise here as the limits of the dual polynomials. However, they share another connection: they both make up the terms of a Sturm-Liouville basis. We will explore this connection further in Section \ref{sec:airy_bessel_sturm}.
\par 

\subsection{Our Methods}
Our proof of Theorem \ref{thm:main_hard_edge} will roughly proceed according to the method established \cite{GK}. In Section \ref{sec:zero_temp}, we formally define the processes at hand and compute the joint distribution of the mean zero Gaussians $\xi_{a,k}$. In Section \ref{sec:jumping_process}, we show how to decompose the covariances of the Laguerre $\infty$-corners process in terms of dual polynomials, and we give a contour integral representation of said polynomials to make asymptotic analysis possible. Finally, in Section \ref{sec:hard_edge_limit}, we perform a steepest analysis of the contour to establish the main result.
\par
Two main sources of difficulty arise when applying this method that are not present in the work of \cite{GK}. The first of these emerges as the matrices $X_k^*X_k$ no longer have distinct eigenvalues for integers $k>N$. This breaks the assumptions used to determine the joint distribution of corners processes and to diagonalize the relevant transition matrices. We resolve the former issue through an analysis of characteristic polynomials and the application of a classical determinant identity. For the latter, we show the transition matrices can still be diagonalized in the same manner. However, we need to extend these matrices to account for the repeated roots and the corresponding contour based proofs of \cite{GK} must be modified to deal with additional residues. We also remark that our decomposition of the covariances differs from \cite{GK} as we incorporate the results of \cite{AHV} to avoid an infinite summation.
\par 
The other source of difficulty arises in obtaining contour integral expressions for the dual polynomials and performing steepest descent. We show that by exploiting the dual polynomial recurrence one can obtain two separate contour expressions in the Laguerre case. As the Hermite polynomials can be evaluated from the Laguerre polynomials $L_n^{(\alpha)}(x)$ with $\alpha = \pm\frac{1}{2}$, one of these expressions is a natural extension of the one obtained in \cite{GK}. Unfortunately, for general $\alpha$, a steepest descent analysis of this expression proved too complex. The other expression we obtained has its own roadblock as its integrand has an essential singularity at 1. This essential singularity is the primary reason this paper only concerns the hard edge limit and not the soft edge limit as we were unable to find a way of adjusting the contour to make steepest descent possible. However, we prove that, at the hard edge, one can select a suitable contour that asymptotically approaches the essential singularity as the number of rows tends to infinity.
\par 
Theorem \ref{thm:secondary_polymer_result} will altogether have a much shorter proof. The computation of the transition probabilities $P^{v_1, v_2}(a \to s)$ will readily follow from techniques used in the proof of Theorem \ref{thm:main_hard_edge}. From here, the main tools we need to complete the proof will be a Fourier-like expansion arising from the aforementioned Sturm-Liouville basis and a Mittag-Leffler expansion for Bessel functions.

\subsection*{Acknowledgements} The author would like to thank Vadim Gorin for helpful discussions and Alexei Borodin for his guidance and feedback throughout the research process. The author was supported by the NSF Graduate Research fellowship under grant $\#$1745302 and partially by Alexei Borodin's grants NSF DMS-1853981 and the Simons Investigator program.

\section{The Laguerre $\beta$-Ensemble and Extensions}
\label{sec:zero_temp}

\subsection{The Laguerre $\beta$-Corners Process}
Let $X_1,X_2,X_3,\ldots$ be an infinite sequence of random matrices whose entries are independent standard normal Gaussians. These entries can be real, complex, or quaternionic corresponding to a parameter $\beta$ equaling 1,2, or 4 respectively. We assume that the matrix $X_k$ is $N \times k$ and forms the leftmost block of $X_{k+1}$ for all $k$. Define $\lambda^k = (\lambda_{1,k}, \ldots, \lambda_{N \wedge k,k})$ to be the nonzero squared singular values of $X_k$ listed in decreasing order where $N \wedge k = \min(N,k)$. For the purpose of simplifying some future equations, we also set $\lambda_{i,k} = 0$ for $ N \wedge k < i \le k$.
\par 
For each $k$, the distribution of $\lambda^k$ is given by the Laguerre $\beta$-ensemble with $N \wedge k$ particles and $\alpha = |N-k|$ \cite{For10}. The embedding of $X_k$ inside of $X_{k+1}$ allows one to extend this one dimensional ensemble to a multi-dimensional one by considering the rows $\lambda^k$ jointly across multiple values of $k$. In particular, for $k < N$, the top $k \times k$ corner of $X_{k+1}^*X_{k+1}$ is $X_k^*X_k$; for $k \ge N$, $X_{k+1}X_{k+1}^*$ is a rank 1 perturbation of $X_kX_k^*$. In both cases this implies that the squared singular values $\lambda^k, \lambda^{k+1}$ interlace \cite{Ner, For22}.
\par 
Combining these properties with an analysis of the characteristic polynomial of $X_k^*X_k$ allows one to compute this joint distribution as follows:
\begin{proposition} For $\beta = 1,2,4$ and $n \ge N$, the joint distribution of $\lambda^1,\ldots, \lambda^n$ has density proportional to
\begin{equation} \label{eqn:beta_laguerre_corners_dist}
\prod_{i=1}^N \lambda_{i,N}^{\frac{\beta}2-1}e^{-\frac{\beta\lambda_{i,n}}2}\prod_{1 \le i < j \le N} (\lambda_{i,n} - \lambda_{j,n}) \prod_{k=1}^{n-1}\prod_{1 \le i < j \le N \wedge k} (\lambda_{i,k} - \lambda_{j,k})^{2-\beta} \prod_{k=1}^{n-1}\prod_{a = 1}^{k \wedge N}\prod_{b=1}^{k+1 \wedge N} |\lambda_{a,k} - \lambda_{b,k+1}|^{\frac{\beta}2 - 1}.
\end{equation}
\end{proposition}
\begin{proof}
We'll begin by dealing with the case $n = N$. Note that the joint distribution of $\lambda^1,\ldots, \lambda^N$ is invariant under left multiplication of $X_N$ by a unitary matrix. As $\lambda^N$ is fixed by this action, the joint distribution of $\lambda^{1},\ldots,\lambda^{N-1}$ conditioned on $\lambda^N$ is also invariant under the action. This invariance property uniquely determines the joint distribution to be (see \cite{Ner} for a proof)
\begin{equation} \label{eq:conditional_probability_lower_rows}
    \mathbb{P}(\lambda^1,\ldots, \lambda^{N-1}|\lambda^N) = \frac{1}{Z_N}\prod_{k=1}^{N-1}\left[\prod_{1 \le i < j \le k} (\lambda_{i,k} - \lambda_{j,k})^{2-\beta} \prod_{a=1}^k \prod_{b=1}^{k+1} |\lambda_{a,k} - \lambda_{b,k+1}|^{\frac{\beta}2-1} \right],
\end{equation}
where
\[Z_N = \prod_{k=1}^N\frac{\Gamma(\beta/2)^k}{\Gamma(k\beta/2)}\prod_{1 \le i < j \le N} (\lambda_{i,N} - \lambda_{j,N})^{\beta-1}.\]
As $\lambda^N$ is distributed according to the Laguerre $\beta$-ensemble with $\alpha = 0$, the case $n=N$ immediately follows. To deal with $n > N$ we will determine the conditional probabilities $\bb{P}(\lambda^{n+1}|\lambda^n)$. For $\beta = 1,2$, these were computed in \cite{FR}. To prove the result generally, we employ the following identity also used in \cite{FR}:
\begin{equation} \label{eq:determinant_equality}
\det(I_{K\times K} - A_{K\times M}B_{M\times K}) = \det(I_{M\times M} - B_{M\times K}A_{K \times M}).
\end{equation}
Let $c_{n+1} = (c_{n+1}^1,\ldots,c_{n+1}^N)^{\top}$ be the final column of $X_{n+1}$ and let $p_{n+1}(z)$ be the characteristic polynomial of $X_{n+1}X_{n+1}^*$. Consider $p_{n+1}(z)$ as a random function conditioned on $\lambda^n$. We have:
\begin{align} p_{n+1}(z) &= \det(zI_{N\times N} - X_nX_n^* - c_{n+1}c_{n+1}^*) \nonumber \\
&=p_n(z) \det(I_{N\times N} - c_{n+1}c_{n+1}^*(zI_{N\times N}-X_nX_n^*)^{-1}) \nonumber \\
&=p_n(z) \det(I_{1\times 1} - c_{n+1}^*(zI_{N\times N}-X_nX_n^*)^{-1}c_{n+1}) \nonumber \\
&\,{\buildrel d \over =}\ p_n(z) \left(1 - \sum_{i=1}^N \frac{|c_{n+1}^i|^2}{z-\lambda_{i,n}}\right). \label{eq:char_poly_equivalence}
\end{align}
where we employed \eqref{eq:determinant_equality} in the third line. Additionally, the final equality here is only in distribution and holds as unitary matrices, such as the one diagonalizing $X_nX_n^*$, preserve the distribution of Gaussians, such as $c_{n+1}^*$.
Plugging in $\lambda_{j,n}$ then gives the distributional equality 
\begin{equation} \label{eq:phi_defn}
\phi(\lambda^{n+1}) := \left(-\frac{p_{n+1}(\lambda_{1,n})}{p_{n}'(\lambda_{1,n})},\ldots, -\frac{p_{n+1}(\lambda_{N,n})}{p_{n}'(\lambda_{N,n})}\right) \,{\buildrel d \over =} (|c_{n+1}^1|^2,\ldots, |c_{n+1}^N|^2).
\end{equation}
The approach from here on out will be to use the map $\phi$ to transform the distribution of the independent $\frac{1}{\beta}\chi^2_{\beta}$ random variables $|c_{n+1}^i|^2$ to the distribution of $\lambda^{n+1}$. As $\phi$ is injective such a transformation is valid, and thus we get the following:
\begin{equation} \label{eq:conditional_probability_exp1}
\mathbb{P}(\lambda^{n+1}|\lambda^n) = \frac{(\beta/2)^{\frac{N\beta}2}}{\Gamma(\beta/2)^N} \left(\prod_{i=1}^N -\frac{p_{n+1}(\lambda_{i,n})}{p_{n}'(\lambda_{i,n})}\right)^{\frac{\beta}2-1} \exp\left(\frac{\beta}{2}\sum_{i=1}^N\frac{p_{n+1}(\lambda_{i,n})}{p_{n}'(\lambda_{i,n})}\right)|\det(D\phi)|.\end{equation}
The product term is easily expanded out in terms of the roots $\lambda_{j,n+1}$. By comparing the degree $N-1$ terms of \eqref{eq:char_poly_equivalence}, the exponential term can be simplified to
\[\exp\left(\frac{\beta}{2}\sum_{i=1}^N (\lambda_{i,n} - \lambda_{i,n+1})\right).\]
Finally, for the Jacobian, note for all $1 \le i, j \le N$ we have
\[\frac{\partial}{\partial \lambda_{j,n+1}} \frac{p_{n+1}(\lambda_{i,n})}{p_{n}'(\lambda_{i,n})} = \frac{1}{\lambda_{j,n+1} -\lambda_{i,n}} \frac{p_{n+1}(\lambda_{i,n})}{p_{n}'(\lambda_{i,n})}.\]
The Cauchy determinant identity then yields
\begin{equation} \label{eq:jacobian_evaluation}
|\det(D\phi)| =  \left|\frac{\prod\limits_{1 \le i < j \le N} (\lambda_{i,n} - \lambda_{j,n})(\lambda_{i,n+1} - \lambda_{j,n+1})}{\prod\limits_{i,j=1}^N (\lambda_{i,n} - \lambda_{j,n+1})} \prod_{i=1}^N \frac{p_{n+1}(\lambda_{i,n})}{p_{n}'(\lambda_{i,n})} \right|= \prod_{1\le i < j \le N}\left|\frac{\lambda_{i,n+1} - \lambda_{j,n+1}}{\lambda_{i,n} - \lambda_{j,n}}\right|.
\end{equation}
Piecing everything together, \eqref{eq:conditional_probability_exp1} simplifies to
\begin{equation} \label{eq:conditional_probability_higher_rows}
\mathbb{P}(\lambda^{n+1}|\lambda^n) = \frac{(\beta/2)^{\frac{N\beta}2}} {\Gamma(\beta/2)^N} \cdot \frac{\prod\limits_{1\le i < j \le N} |\lambda_{i,n+1} - \lambda_{j,n+1}| \prod\limits_{a,b = 1}^N|\lambda_{a,n} - \lambda_{b,n+1}|^{\frac{\beta}2 - 1} \prod\limits_{i=1}^N e^{-\frac{\beta \lambda_{i,n+1}}2}}{\prod\limits_{1 \le i < j \le N} |\lambda_{i,n} - \lambda_{j,n}|^{\beta - 1} \prod\limits_{i=1}^N e^{-\frac{\beta \lambda_{i,n}}2}},
\end{equation}
The desired result now follows by induction.
\end{proof}
\begin{definition} (Laguerre $\beta$-Corners Process) Let $\beta$ be any positive parameter and suppose $n \ge N$ are two positive integers. The joint distribution of the eigenvalues $\lambda^1,\ldots, \lambda^n$ with density given by \eqref{eqn:beta_laguerre_corners_dist} will be referred to as the \textit{Laguerre $\beta$-corners process with $n$ rows centered at $N$ }. We will typically exclude ``centered at $N$'' for the sake of brevity and reserve the variable $N$ for this purpose to avoid confusion.
\end{definition}
Using the interpretation of the random variables $\lambda_{i,k}$ as eigenvalues, it is immediate for $\beta = 1,2,4$ that the restriction of \eqref{eqn:beta_laguerre_corners_dist} to a single level (i.e. $\lambda^k$ for some positive integer $k$) is given by the Laguerre $\beta$-ensemble. To extend this to all $\beta > 0$, we need the Dixon-Anderson integration identity. This identity was proven in \cite{And91} in the context of Selberg integrals, but can also be derived as a limit of a much older identity of Dixon \cite{Dix}.
\begin{lemma} (Dixon-Anderson) \label{lem:dixon-anderson} The following holds:
\begin{equation*} 
\int \cdots \int \prod_{1 \le i < j \le m} (u_i - u_j) \prod_{i,j} |u_i - v_j|^{s_j - 1}du_1\cdots du_m = \frac{\prod_{j=1}^{m+1} \Gamma(s_j)}{\Gamma\left(\sum_{i=1}^{m+1}s_i\right)} \prod_{1 \le i < j \le m+1} (v_i - v_j)^{s_i + s_j - 1},
\end{equation*}
where integration is performed over reals satisfying $v_i > u_i > v_{i+1}$ for $i = 1,\ldots, m$.
\end{lemma}
\begin{corollary} For all $\beta > 0$ positive integers $k \le n$, the restriction of the Laguerre $\beta$-corners process with $n$ rows to its $k$-th level is given by the Laguerre $\beta$-ensemble with $N \wedge k$ values and parameter $\alpha = |N-k|$.
\end{corollary}
\begin{proof}
This follows from inductively applying the Dixon-Anderson identity with $u_i = \lambda_{i,m}$, $v_i = \lambda_{i,m+1}$, and $s_i = \beta/2$ for all relevant $i$.
\end{proof}

\subsection{Zero Temperature Limit}
Ultimately, our interest will be in analyzing the covariances of Laguerre $\beta$-corners process, under the successive limit regimes $\beta \to \infty$, $N \to \infty$. This section will be devoted to obtaining an explicit description of the former limit. As it turns out this description has a close relation to the roots of the generalized Laguerre polynomials.
\par 
The generalized Laguerre polynomial\footnote{These are sometimes referred to as the \textit{associated Laguerre polynomials} in the literature--not to be confused with the associated polynomials defined in Section \ref{sec:Q_m_properties}}  $L^{(\alpha)}_n(x)$ is defined to be the polynomial solution to
\begin{equation} 
xy'' + (\alpha + 1 - x)y' + ny = 0,
\end{equation}
with leading coefficient $\frac{(-1)^n}{n!}$. For $\alpha \ge 0$, $\{L^{(\alpha)}_i(x)\}_{i \in \N}$ equivalently are the orthogonal polynomials with respect to the weight $x^{\alpha}e^{-x}$. In the case $\alpha < 0$, one can show
\begin{equation} \label{eq:laguerre_neg_alpha}
L^{(-\alpha)}_{n+\alpha}(x) = \frac{n!(-x)^{\alpha}}{(n+\alpha)!}L_{n}^{(\alpha)}(x).
\end{equation}
Setting $P_k(x) = (-1)^k k!L^{(N-k)}_k(x)$ to be the monic version, one also obtains the identity $P_{k}'(x) = kP_{k-1}(x)$. Sequences of polynomials satisfying this property are referred to as \textit{Appell sequences} and they will play an important role in Section 3 of this paper. For the Laguerre polynomial identities used here and throughout this paper see \cite{Sze} for reference. Finally, we define the Laguerre $\infty$-corners process which will be a central distribution in the remainder of the paper:
\begin{definition} (Laguerre $\infty$-Corners Process) The \textit{Laguerre $\infty$-corners process centered at $N$} is the distribution on sets of real numbers $\{\xi_{a,k}\}_{1 \le a \le k \le n}$ with $\xi_{a,k} = 0$ deterministically whenever $a > N$ and the remaining values distributed such that they have joint density proportional to:
\begin{equation} \label{eq:infinity_corners_dist}
\exp\left[-\frac{1}{4}\sum_{i=1}^n \left(\frac{\xi_{i,n}}{l_{i,n}}\right)^2 + \sum_{k=1}^{n-1}\left( \frac{1}{2}\sum_{1 \le i < j \le N \wedge k} \left(\frac{\xi_{j,k} - \xi_{i,k}}{l_{j,k} - l_{i,k}}\right)^2 - \frac{1}{4}\sum_{a=1}^{k \wedge N} \sum_{b=1}^{k+1 \wedge N} \left(\frac{\xi_{a,k} - \xi_{b,k+1}}{l_{a,k} - l_{b,k+1}}\right)^2 \right)\right],
\end{equation}
where $l_{i,k}$ is the $i$-th largest root of $L_k^{(N-k)}(x)$. It will be shown in the proof of Proposition \ref{prop:laguerre_infinity_corners} that this is the density of a zero mean multivariate Gaussian. As before we will typically exclude ``centered at $N$'' and refer to $\xi^k = (\xi_{1,k},\ldots, \xi_{k,k})$ as a \textit{level} of the ensemble. Furthermore, in the remainder of the paper when discussing the density of the Laguerre $\infty$-Corners process we will in actuality be referring to the density of the nonzero values. The presence of the zero values is for notational convenience such as in Corollary \ref{cor:xi_k_conditional_top_level}.\footnote{To avoid confusion, we note that \cite{GM} refers to the pair containing both the crystallized values and the Gaussian fluctuations as an $\infty$-\textit{corners process} rather than just the fluctuations.}
\end{definition}
\par 
With these definitions, we can now describe the $\beta \to \infty$ limit of the Laguerre $\beta$-corners process. The overarching message of the following proposition is that the Laguerre $\beta$-corners process crystallizes at roots of the Laguerre polynomials with fluctuations described by the Laguerre $\infty$-corners process. Analogous work for the Gaussian/Hermite case can be found in \cite{GM}. 

\begin{proposition} \label{prop:laguerre_infinity_corners} (Laguerre $\infty$-Corners) Suppose $\{\lambda_{i,k}\}_{1 \le i \le k \le n}$ is distributed according to the Laguerre $\beta$-corners process with $n$ rows. As $\beta \to \infty$, the distribution of the eigenvalues $\{\lambda_{i,k}\}_{1 \le i \le k \le n}$ converges weakly to a delta measure on the deterministic array $\{l_{i,k}\}_{1 \le i \le k \le n}$ where $l_{i,k}$ is the $i$-th largest root of $L^{(N-k)}_k(x)$. Furthermore, the distribution of the set
\[\{\sqrt{\beta}(\lambda_{i,k} - l_{i,k})\}_{1 \le i \le k \le n}\]
of fluctuations away from this array converges weakly to the Laguerre $\infty$-corners process with $n$ rows. 
\end{proposition}
\begin{proof}[Proof of Crystallization]
As $\beta \to \infty$, the terms in the Laguerre $\beta$-corners density \eqref{eqn:beta_laguerre_corners_dist} without $\beta$ dependence become negligible and the distribution converges weakly to a delta measure at
\begin{equation}  \label{eq:crystal_argmax_full}
\argmax \prod_{i=1}^N \lambda_{i,N}^{1/2}e^{-\frac{\lambda_{i,n}}2}\prod_{k=1}^{n-1} \prod_{1 \le i < j \le N \wedge k} |\lambda_{i,k}-\lambda_{j,k}|^{-1} \prod_{k=1}^{n-1}\prod_{a=1}^{k \wedge N} \prod_{b=1}^{k+1 \wedge N} |\lambda_{a,k} - \lambda_{b,k+1}|^{1/2}
\end{equation}
Let $y_{i,k}$ be the value of $\lambda_{i,k}$ under this argmax (our arguments will imply it is unique for all relevant $i,k$) and define $P_{k}(x)$ to be the monic polynomial with roots $y_{1,k},\ldots, y_{k\wedge N,k}$ for each $k \le n$. It follows from the proposition statement and \eqref{eq:laguerre_neg_alpha}, that we want $P_k(x) = (-1)^k k!L^{(N-k)}_k(x)$ for $k \le N$ and $P_{k}(x) = (-1)^N N! L^{(k-N)}_N(x)$ for $k \ge N$. These desired values satisfy the recursive identities $P_k'(x) = kP_{k-1}(x)$ and $P_{k+1}(x) = P_k(x) - P_{k}'(x)$ respectively. To prove the result then, it suffices to show these identities along with $P_N(x) = L^{(0)}_N(x)$.
\par 
In \cite[sec. 3.4]{GM} on the crystallization of corners processes, the authors showed the first recursive identity holds by equating the logarithmic derivatives of \eqref{eq:crystal_argmax_full} with zero. Furthermore, the same paper shows that
\[\prod_{k=1}^{N-1} \prod_{1 \le i < j \le N \wedge k} |y_{i,k}-y_{j,k}|^{-1} \prod_{k=1}^{N-1}\prod_{a=1}^{k \wedge N} \prod_{b=1}^{k+1 \wedge N} |y_{a,k} - y_{b,k+1}|^{1/2} = \prod_{1 \le i < j \le N} |y_{i,N}-y_{j,N}|.\]
Thus \eqref{eq:crystal_argmax_full} reduces to evaluating 
\begin{equation}  \label{eq:crystal_argmax_simplified}
\argmax \prod_{i=1}^N \lambda_{i,N}^{1/2}e^{-\frac{\lambda_{i,n}}2}\prod_{k=N+1}^{n-1} \prod_{1 \le i < j \le N \wedge k} |\lambda_{i,k}-\lambda_{j,k}|^{-1} \prod_{k=N}^{n-1}\prod_{a=1}^{k \wedge N} \prod_{b=1}^{k+1 \wedge N} |\lambda_{a,k} - \lambda_{b,k+1}|^{1/2}.
\end{equation}
Taking logarithmic derivatives at $\lambda_{i,n}$ for $i \le N$, we get 
\begin{equation} \label{eq:top_row_log_deriv} 0 = -\frac{1}{2} + \frac{1}{2}\sum_{a = 1}^{n} \frac{1}{y_{i,n} - y_{a,n-1}} = -\frac{1}{2} + \frac{P_{n-1}'(y_{i,n})}{2P_{n-1}(y_{i,n})},
\end{equation}
which implies the second recursion at $k = n-1$. For other $k$ we proceed by a backward induction. Taking logarithmic derivatives at $\lambda_{i,k}$ again for $i \le N$ we get
\begin{align*}
    0  &= -\sum_{j \neq i} \frac{1}{y_{i,k} - y_{j,k}} + \frac{1}{2}\sum_{a=1}^N \frac{1}{y_{i,k} - y_{a,k+1}} + \frac{1}{2}\sum_{a=1}^N \frac{1}{y_{i,k} - y_{a,k-1}} \\
    &= - \frac{P_k''(y_{i,k})}{2P_k'(y_{i,k})} + \frac{P_{k+1}'(y_{i,k})}{2P_{k+1}(y_{i,k})} + \frac{P_{k-1}'(y_{i,k})}{2P_{k-1}(y_{i,k})}.
\end{align*}
Substituting $P_{k+1}(y_{i,k}) = P_k(y_{i,k}) - P_k'(y_{i,k})$ reduces this to the same form as \eqref{eq:top_row_log_deriv}, which completes the induction. Finally, the logarithmic derivative at $\lambda_{i,N}$ for $i \le N$ yields
\[0 = \frac{1}{2y_{i,N}} + \frac{1}{2}\sum_{a=1}^N \frac{1}{y_{i,N} - y_{a,N+1}} = \frac{1}{2y_{i,N}} + \frac{P_{N+1}'(y_{i,N})}{2P_{N+1}(y_{i,N})} = \frac{1}{2y_{i,N}} + \frac{P_{N}'(y_{i,N}) - P_{N}''(y_{i,N})}{-2P_{N}'(y_{i,N})},\]
which is equivalent to the differential equation \eqref{eq:laguerre_diffeq} defining $L_N^{(0)}(x)$.
\end{proof}
\begin{proof}[Proof of Gaussian Fluctuations]
Write the density \eqref{eqn:beta_laguerre_corners_dist} in the form
\[\frac{1}{Z(\beta)}\exp\left(\beta h_1(\lambda) - h_2(\lambda)\right),\]
where $Z(\beta)$ is the normalization constant,
\[h_1(\lambda) = \frac{1}{2}\sum_{i=1}^N (\log\lambda_{i,N}-\lambda_{i,n}) - \sum_{k=1}^{n-1}\sum_{1 \le i < j \le N \wedge k} \log|\lambda_{i,k} - \lambda_{j,k}| + \frac{1}{2} \sum_{k=1}^{n-1}\sum_{a = 1}^{k \wedge N}\sum_{b=1}^{k+1 \wedge N} \log|\lambda_{a,k} - \lambda_{b,k+1}|,\]
and $h_2(\lambda)$ contains the remaining terms. Now, consider the Taylor expansion of the density about $l_{i,k}$ by setting
\[\lambda_{i,k} = l_{i,k} + \frac{1}{\sqrt{\beta}}\xi_{i,k}. \]
We have
\begin{align} h_1(\lambda) - h_1(l) =& \frac{1}{2}\sum_{i=1}^N \left(\log\left(1 + \frac{\xi_{i,N}}{l_{i,N}\sqrt{\beta}}\right)-\frac{\xi_{i,n}}{\sqrt{\beta}}\right) - \sum_{k=1}^{n-1}\sum_{1 \le i < j \le N \wedge k} \log\left|1 + \frac{\xi_{i,k} - \xi_{j,k}}{\sqrt{\beta}(\lambda_{i,k} - \lambda_{j,k})}\right| \nonumber \\
&+ \frac{1}{2} \sum_{k=1}^{n-1}\sum_{a = 1}^{k \wedge N}\sum_{b=1}^{k+1 \wedge N} \log\left|1 + \frac{\xi_{a,k} - \xi_{b,k+1}}{\sqrt{\beta}(\lambda_{a,k} - \lambda_{b,k+1})}\right|. \label{eq:fluctuations_xi_dependent_terms}
\end{align}
In the Taylor expansion of the right-hand side of \eqref{eq:fluctuations_xi_dependent_terms}, the first order terms have coefficients given by the logarithmic derivatives determining \eqref{eq:crystal_argmax_full}. Hence, they vanish. Let $g(\xi)$ be the second order terms and note $\exp(\beta g(\xi))$ is precisely \eqref{eq:infinity_corners_dist}.
\par 
Now, suppose we knew that \eqref{eq:infinity_corners_dist} is the distribution of a multivariate Gaussian; that is, $g(\xi) < 0$ whenever $\xi \neq 0$. Then this Taylor expansion yields
\[
\beta(h_1(\lambda) - h_1(l)) = g(\xi) + O(\beta^{-1/2}),
\]
where the bound is uniform whenever $\xi$ lies in a compact subset $C$. Under the same conditions, $h_2(y) = h_2(l) + O(\beta^{-1/2})$, allowing us to write the density as
\[
\exp(\beta h_1(l) + h_2(l))Z(\beta)^{-1}\exp(g(\xi) + O(\beta)^{-1/2}).
\]
As this has $L^1$ norm equal to 1 and decays rapidly as $|\xi|\to\infty$, the normalization $\exp(\beta h_1(l) + h_2(l))Z(\beta)^{-1}$ must converge to a constant. From here, for any $f$ with compact support, dominated convergence implies
\[\lim_{\beta \to \infty} \exp(\beta h_1(l) + h_2(l))Z(\beta)^{-1}\int f(\xi)\exp(g(\xi) + O(\beta)^{-1/2})d\xi = \frac{1}{Z}\int f(\xi) \exp(g(\xi))d\xi,\]
where $Z \neq 0$ is some constant. Hence, convergence holds.
\par
To see that \eqref{eq:infinity_corners_dist} is in fact the distribution of a well-defined multivariate Gaussian, we use the same integration identity that \cite{GM} used in the Hermite case:
\begin{equation} 
\int\cdots \int \exp\left(-\sum_{a=1}^k\sum_{b=1}^{k-1} \left(\frac{\zeta_{a}^k - \zeta_{b}^{k-1}}{x_a^k - x_b^k}\right)^2\right)d\zeta_{1}^{k-1}\cdots d\zeta_{k-1}^{k-1} = \frac{1}{Z'}\exp\left(-2\sum_{1 \le i < j \le k}\left(\frac{\zeta_i^k - \zeta_j^k}{x_{i}^k-x_j^k}\right)^2\right),
\end{equation}
where $k$ is some positive integer, $\zeta_1^k, \ldots, \zeta_k^k$ and $x_1^k > x_{2}^k > \cdots > x_k^k$ are reals, and $Z'$ is some normalizing constant independent of the $\zeta^k_i$. Using this identity, we can iteratively compute \eqref{eq:infinity_corners_dist} by integrating against $\xi^1, \xi^2,\ldots, \xi^{n-1}$ successively. In particular, for $k \le N$ we set $\zeta^{k-1}, \zeta^{k} = \xi^{k-1}, \xi^{k}$ and $x_i^{k-1}, x_i^{k} = l_{i,k-1}, l_{i,k}$ respectively for all $i \le k$. For $k > N$, we set $\zeta^N_i = \xi_{i,k-1}, \zeta^{N+1}_i = \xi_{i,k}, x_i^N = l_{i,k-1}, x_i^{N+1} = l_{i,k}$ for all $i \le N$ and then take $x_{N+1}^{N+1} \to -\infty$.
\end{proof}

\begin{corollary} \label{cor:xi_k_conditional_top_level}
Let $\{\xi_{i,k}\}_{1 \le i \le k \le n}$ be distributed according to the Laguerre $\infty$-corners process with $n$ rows. The distribution of $\xi^k$ conditioned on $\xi^{k+1}$ has density proportional to
\[\exp\left(-\frac{1}{4} \sum_{a=1}^{N \wedge k} \sum_{b=1}^{k+1} \left(\frac{\xi_{a,k} - \xi_{b,k+1}}{l_{a,k} - l_{b,k+1}}\right)^2\right).\]
Furthermore, the top level $\xi^n$ has density proportional to
\[\exp\left(-\frac{n - N + 1}{4} \sum_{i=1}^{N} \left(\frac{\xi_{i,n}}{l_{i,n}}\right)^2 - \frac{1}{2}\sum_{1 \le i < j \le N} \left(\frac{\xi_{i,n} - \xi_{j,n}}{l_{i,n} - l_{j,n}}\right)^2\right).\]
\end{corollary}
\begin{proof}
The conditional distribution follows from applying the Taylor expansion $\lambda_{i,k} = l_{i,k} + \frac{1}{\sqrt{\beta}}\xi_{i,k}$ to equations \eqref{eq:conditional_probability_lower_rows}, \eqref{eq:conditional_probability_higher_rows}. The density of the top row follows from integrating out the first $n-1$ rows of \eqref{eq:infinity_corners_dist} using the procedure described at the end of the proof of Proposition \ref{prop:laguerre_infinity_corners}.
\end{proof}

\section{Jumping Process and Orthogonal Polynomials}
\label{sec:jumping_process}
In this section, we introduce the main tools we will use to analyze the hard edge limit. Many of the techniques we apply here have their origin in \cite{GK}. The key insight of that paper was that $\beta = \infty$ limit of many corners processes can be represented in terms of a certain random walk on the roots of an Appell sequence.
\par 
Following \cite[sec. 3]{GK}, we will begin by showing how such a random walk on the roots of Laguerre polynomials arises from the Laguerre $\infty$-corners process. After this, Section 3.2 will review the techniques established in \cite[sec. 4]{GK} for diagonalizing the transition matrix of these random walks in terms of a certain family of orthogonal polynomials. Unfortunately, our case necessitates us to consider Appell sequences consisting of polynomials with repeated roots, whereas the previous work assumes distinct roots. As a result, Section 3.3 will be devoted to demonstrating how the proofs of \cite{GK} can be extended to our case through some extra tracking of residues. In Section 3.4, we will exploit the connection between the orthogonal polynomials and dual polynomials to prove several useful properties including two contour integral representations of said orthogonal polynomials. Finally, in Section 3.5, we will combine this work along with a theorem of \cite{AHV} to demonstrate how one can decompose the covariances of the Laguerre $\infty$-corners process in terms of the orthogonal polynomials.

\subsection{Additive Polymer}
\label{sec:additive_polymer}
For the remainder of this paper, $\{\xi_{i,k}\}_{1 \le i \le k \le n}$ will be assumed to be distributed according to the Laguerre $\infty$-corners process with $n$ rows and $l_{i,k}$ will be the $i$-th largest root of $L_k^{(N-k)}(x)$. Consider the conditional distribution of $\xi^k$ given $\xi^{k+1}$ from Corollary \ref{cor:xi_k_conditional_top_level} for some $k \le n-1$. Completing the squares gives that this density is proportional to
\[\prod_{a=1}^{N \wedge k}\exp\left(-\frac{1}{4}\left(\sum_{b' = 1}^{k+1} (l_{a,k} - l_{b',k+1})^{-2} \right)\left(\xi_{a,k} - \sum_{b=1}^{k+1}\xi_{b,k+1}\frac{(l_{a,k} - l_{b,k+1})^{-2}}{\sum_{b' = 1}^{k+1} (l_{a,k} - l_{b',k+1})^{-2}}\right)^2\right),\]
where the normalization constant may depend on $\xi^{k+1}$. 
Let $P_k(x) = (-1)^kk!L_{k}^{(N - k)}(x)$. For any $a \le N \wedge k$, we evaluate
\begin{equation} \label{eq:weight_evaluation}
\sum_{b' = 1}^{k+1} (l_{a,k} - l_{b',k+1})^{-2} = -\frac{P_{k+1}''(l_{a,k})}{P_{k+1}(l_{a,k})} = \frac{k+1}{l_{a,k}},
\end{equation}
where the second equality follows from the Laguerre polynomial differential equation \eqref{eq:laguerre_diffeq}.
Now let:
\begin{equation} \label{eq:laguerre_trans_prob_one_level}
\alpha_{a,b}^{k} = \begin{cases} \displaystyle\frac{(l_{a,k} - l_{b,k+1})^{-2}}{\sum_{b' = 1}^{k+1} (l_{a,k} - l_{b',k+1})^{-2}} & a \le N \\
\\
\displaystyle\frac{\1(b > N)}{k - N} & a > N.
\end{cases}
\end{equation}
whenever $a \le k, b \le k+1$. The case $a > N$ will not actually matter for our purposes here, but will prove useful in the following subsections. Using this new terminology we can rewrite the density in the simple form
\[\prod_{a=1}^{N \wedge k}\exp\left(-\frac{1}{2}\left( \frac{2l_{a,k}}{k+1} \right)^{-1}\left(\xi_{a,k} - \sum_{b=1}^{k+1}\alpha_{a,b}^k\xi_{b,k+1} \right)^2\right).\]
Thus we derive \eqref{eq:intro_to_transition_matrices} for the Laguerre ensemble:
\begin{equation} \label{eq:xi_decomposition_single_level}
\xi_{a,k} = \sum_{b=1}^{k+1}\alpha_{a,b}^k\xi_{b,k+1}  + \eta_{a,k},
\end{equation}
where $\eta_{a,k}$ are independent mean zero Gaussians with
\begin{equation} \label{eq:eta_variance}
\var(\eta_{a,k}) = \frac{2l_{a,k}}{k+1}.
\end{equation}
Note for $a > N$, the Gaussian $\eta_{a,k}$ is degenerate. For each $k \le n$, define $\mathcal{X}_k$ to be the set of roots of $P_k(x)$ counted with multiplicity. The $k \times (k+1)$ matrix $A_k$ with elements $\alpha_{a,b}^k$ is then stochastic and maps $\mathcal{X}_k$ to $\mathcal{X}_{k+1}$. Setting $\eta^k = (\eta_{1,k},\ldots \eta_{k,k})$, \eqref{eq:eta_variance} becomes: 
\[\xi^{k} = A_k\xi^{k+1} + \eta^k.\]
Iterating this yields
\[\xi^{k} = (A_k\cdots A_{n-1})\xi^n + \sum_{j = k}^{n-1}(A_k \cdots A_{j-1})\eta^j.\]
Finally, define the \textit{diffusion kernels}
\[K^{k,l}(a \to b) = (A_k\cdots A_{l-1})_{a,b}.\]
We can summarize all of this work as follows:
\begin{proposition} \label{prop:laguerre_infinity_additive_polymer} Let $\{\xi_{i,k}\}_{1 \le i \le k \le n}$ be distributed according to the Laguerre $\infty$-corners process with $n$ rows. There exists an array $\{\eta_{i,k}\}_{1 \le i \le k \le n}$ of independent mean zero Gaussians with variances given by \eqref{eq:eta_variance} such that for any $k \le n$ and $a \le N \wedge k$ one has
\[\xi_{a,k} = \sum_{b=1}^n K^{k,n}(a \to b) \xi_{b,n} + \sum_{j=k}^{n-1}\sum_{b=1}^{j} K^{k,j}(a\to b) \eta_{b,j}.\]
\end{proposition}

\subsection{General Jumping Processes}
\label{sec:general_jumping_processes}
We will now describe a more general framework that the above results fit into. Let $P_0(x), P_1(x),\ldots$ be an Appell sequence and for each $k \ge 0$ let $\mathcal{X}_k$ be the roots of $P_k(x)$ counted with multiplicity. Furthermore, for $x \in \mathcal{X}_k$, let $r_k(x)$ be the multiplicity of $x$ and let $\mathcal{Y}_k$ be the set of $x \in \mathcal{X}_k$ with $r_{k+1}(x) = 0$. All results in this subsection were originally established and proven in \cite{GK} under the condition $r_k(x) = 1$ for all $k \ge 1, x \in \mathcal{X}_k$. Our proofs will focus on extending said results to the general case.

\begin{definition} (Jumping Process) The \textit{jumping process} associated to the Appell sequence $P_0(x),P_1(x),\ldots$ is the Markov process with state space $\mathcal{X}_k$ at time $k$ and transition probability from $x \in \mathcal{X}_k$ to $y \in \mathcal{X}_{k+1}$ given by
\[\mathbb{P}_k(x \to y) = \begin{cases} -(x-y)^{-2}\frac{P_{k+1}(x)}{P_{k+1}''(x)} & r_{k+1}(y) = 1\\  \\ \displaystyle\frac{\1(x = y)}{r_{k+1}(y)}& r_{k+1}(y) > 1.
\end{cases}\]
\end{definition}
The first equality in \eqref{eq:weight_evaluation} implies that these values are indeed well-defined transition probabilities. The case where $r_{k+1}(y) > 1$ is somewhat redundant as it equals the limit of the case $r_{k+1}(y) = 1$ as $y \to x$. The key to diagonalizing the transition matrices of jumping processes is through their action on functions:
\begin{definition} ($\mathcal{F}_k, D_k$) Let $\mathcal{F}_k$ denote the set of functions from $\mathcal{X}_k \to \R$ such that any two equal values of $\mathcal{X}_k$ map to the same value in $\R$. Note that all such functions can be realized by polynomials. Now define $D_k: \mathcal{F}_k \to \mathcal{F}_{k+1}$ by
\[
(D_kf)(y) = \sum_{x \in \mathcal{X}_k} \mathbb{P}_k(x \to y)f(x).
\]
\end{definition}
\begin{proposition}{\cite{GK}} \label{prop:op_poly_preservation} For $k \ge 1$ and each $m = 0,1,\ldots, |\mathcal{Y}_k|-1$ the linear operator $D_k$ preserves the space of polynomials of degree at most $m$. In more detail,
\[
D_kx^m = \left(1 - \frac{m+1}{k+1}\right)x^m + \textit{ (polynomial of degree $\le m$)}.
\]
\end{proposition}
Unfortunately, current technologies are unable to easily analyze the lower degree terms for a general Appell sequence. However, for a subclass of sequences satisfying a differential equation, the operator $D_k$ is diagonalizable in terms of explicit orthogonal polynomials.
\begin{definition} (Inner Product $\langle \cdot , \cdot \rangle_k$) For each $k \ge 1$, define an inner product on $\mathcal{F}_k$ corresponding to the jumping process on $P_0(x), P_1(x),\ldots$ by
\[\langle f, g \rangle_k = \sum_{x \in \mathcal{Y}_k} f(x)g(x)w_k(x),\]
where the weight is
\[w_k(x) = -\frac{P_{k+1}(x)}{P_{k+1}''(x)}.\]
Note $w_k(x)$ has a natural connection to the normalization constant of the jumping process and is easily evaluated in the Laguerre case by \eqref{eq:weight_evaluation}.
\end{definition}
\begin{definition} (Orthogonal Polynomials $Q_m^{(k)}$, $\tilde{Q}_m^{(k)}$) 
For each $k \ge 1$ and $0 \le m \le |\mathcal{Y}_k| - 1$, let $Q_m^{(k)}(x)$ be the monic orthogonal polynomial of degree $m$ with respect to $\langle \cdot, \cdot \rangle_k$. Furthermore, define the \textit{normalized} version of $Q_m^{(k)}(x)$ by
\[
\tilde{Q}_m^{(k)}(x) = \frac{Q_m^{(k)}(x)}{\sqrt{\langle Q_m^{(k)}(x), Q_m^{(k)}(x) \rangle_k}}.
\]
\end{definition}
\begin{proposition}{\cite{GK}} \label{prop:diagonalization_Q} Suppose for all $k \ge 0$, the polynomial $P_k(x)$ satisfies the recurrence
\[
P_k''(x)\alpha_k(x) + P_k'(x)\beta_k(x) + P_k(x) = 0,
\]
 where $\alpha_k(x)$ is polynomial of degree $\le 2$ and $\beta_k(x)$ is a polynomial of degree $\le 1$. Then one has
\[
D_kQ_m^{(k)}(x) = \left(1 - \frac{m+1}{k+1}\right)Q_m^{(k+1)}(x).
\]
\end{proposition}
As long as $P_k(x)$ satisfies the above recurrence, then the orthogonal polynomials $Q_m^{(k)}(x)$ can now be computed via the Gram-Schmidt procedure applied to the monomials $x^m$. In Section \ref{sec:Q_m_properties}, we will see that in the Laguerre case these polynomials can be evaluated even more simply via the theory of dual polynomials.

\subsection{Proofs in the Case of Repeated Roots}
This section will be devoted to giving proofs of Propositions \ref{prop:op_poly_preservation} and \ref{prop:diagonalization_Q}. Note that when $P_k(x) = (-1)^k k! L^{(N-k)}_k(x)$, once $k > N + 1$, $P_{k}(x)$ has multiple roots equal to 0. Thus these extensions of the corresponding results in \cite{GK} are necessary for any Laguerre $\infty$-corners process with number of rows $n > N + 1$.
\begin{proof}[Proof of Proposition \ref{prop:op_poly_preservation}] Consider first the contribution to $(D_kf)(y)$ from $x \in \mathcal{Y}_k$:
\begin{align} \label{eq:sum_to_contour}
\sum_{x \in \mathcal{Y}_k} \mathbb{P}_k(x\to y) f(x) &= -\frac{1}{2\pi i}\int_{\mathcal{Y}_k}\frac{f(z)}{(z-y)^2} \frac{P_{k+1}(z)}{P_{k+1}''(z)} \frac{P_k'(z)}{P_k(z)}dz, \nonumber\\
&= -\frac{1}{2\pi i}\int_{\mathcal{Y}_k}\frac{f(z)}{(z-y)^2} \frac{P_{k+1}(z)}{P_{k+1}'(z)}dz,
\end{align}
where the contour is a union of positively oriented curves around the elements of $\mathcal{Y}_k$ that do not contain the other poles of the integrand at $z = y,\infty$. The pole at $z = y$ is simple with residue $\frac{f(y)}{r_{k+1}(y)}$. As the contribution to $(D_kf)(y)$ from $x \not\in \mathcal{Y}_k$ is precisely $f(y)\left(1 - \frac{1}{r_{k+1}(y)}\right)$, in all cases we can move the contour over the pole at $z = y$ to get
\[
(D_kf)(y) = f(y) - \frac{1}{2\pi i} \int_{\infty} \frac{f(z)}{(z-y)^2} \frac{P_{k+1}(z)}{P_{k+1}'(z)}dz.
\]
From here we can proceed as in the proof of \cite[prop. 4.3]{GK}. Taylor expand to get
\[\frac{1}{(z-y)^2} = \frac{1}{z^2} + 2\frac{y}{z^3} + 3\frac{y^2}{z^4} + \cdots\]
When $f(y) = y^m$, the expression $f(z)\frac{P_{k+1}(z)}{P_{k+1}'(z)}$ will have leading order $\frac{z^{m+1}}{k+1}$. Thus only the terms from the Taylor expansion of order $z^{-i-1}$ for $i \le m$ will contribute to the pole at $\infty$. Comparing the residues from each term in the expansion then gives precisely the desired result.
\end{proof}

To avoid having to re-do too much of the work from \cite{GK} in the following proof, we will use the following lemma. While not stated explicitly, the lemma is obtained during the proof of \cite[thm. 4.9]{GK}.\footnote{In \cite{GK}'s case, $\mathcal{X}_k = \mathcal{Y}_k$ so the statement of the result there has the contours around $\mathcal{X}_k$. However, the argument still works in the general case with this replacement.} %.
\begin{lemma} \label{lem:diagonalization_Q_helper}
Under the conditions of proposition \ref{prop:diagonalization_Q},
\[\oint_{\mathcal{Y}_{k+1}} \oint_{\mathcal{Y}_{k}} \frac{Q_m^{(k)}(z)}{(z-u)^2} \frac{P_{k+1}(z)}{P_k(z)}u^j \frac{P_{k+2}(u)}{P_{k+1}(u)}dzdu = 0\]
for all $0 \le j \le m-1$. The contours are unions of positively oriented curves around the points of $\mathcal{Y}_{k+1}, \mathcal{Y}_{k}$ sufficiently small such that each curve only contains a single pole of its respective integrand.
\end{lemma}
\begin{proof}[Proof of Proposition \ref{prop:diagonalization_Q}] 
Following off of Proposition \ref{prop:op_poly_preservation}, it suffices to show
\[\langle D_k Q_m^{(k)}, x^j \rangle_{k+1} = 0\]
for $0 \le j \le m-1$. Note
\[\langle D_k Q_m^{(k)}, x^j\rangle_{k+1} = \sum_{y \in \mathcal{Y}_{k+1}} \sum_{x \in \mathcal{X}_{k}} \mathbb{P}_k(x \to y) Q_m^{(k)}(y)y^jw_k(j).\]
As $\mathbb{P}_k(x \to y) = 0$ if $y \in \mathcal{Y}_{k+1}, x \not\in \mathcal{Y}_{k}$, the inner sum may be replaced with a summation over $\mathcal{Y}_k$ instead. From here, expanding the sums out as contour integrals as in \eqref{eq:sum_to_contour} gets us the expression from Lemma \ref{lem:diagonalization_Q_helper} thus completing the proof.
\end{proof}

\subsection{Orthogonality via Dual Polynomials} 
\label{sec:Q_m_properties}
Suppose $\mathcal{P}_0(x),\mathcal{P}_1(x),\ldots$ is a sequence of polynomials with distinct roots satisfying the three-term recurrence
\begin{equation} \label{eq:three_term_recurrence}
x\mathcal{P}_i(x) = \mathcal{P}_{i+1}(x) + b_i\mathcal{P}_i(x) + u_i\mathcal{P}_{i-1}(x)
\end{equation}
for all $i \ge 1$ some constants $b_i,u_i$. Any family of monic polynomials orthogonal with respect to a positive measure satisfies such a recurrence (see \cite{Chi} for a primer). Conversely, for any fixed positive integer $M$, the sequence of polynomials $\mathcal{P}_0(x),\ldots, \mathcal{P}_{M-1}(x)$ are orthogonal with respect to the discrete measure
\[
\sum_{i=1}^M w(x_i)\delta_{x_i},
\]
where $x_1,\ldots, x_M$ are the roots of $\mathcal{P}_M(x)$ and
\[
w(x) = \frac{1}{\mathcal{P}_{M-1}(x)\mathcal{P}'_M(x)}.
\]
\par 
For the orthogonal polynomials of Section \ref{sec:general_jumping_processes} corresponding to the Hermite, Jacobi, and Laguerre polynomials, \cite{GK, AHV} both noted that the dual polynomials of de Boor and Saaf \cite{BS} hold the key to evaluating them. Like both papers, we will also use the exposition and results found in \cite{VZ} as a guide.
\begin{definition} (Dual Polynomials) Let $M$ be some fixed positive integer. The \textit{dual polynomials} $\mathcal{Q}_0(x),\ldots, \mathcal{Q}_{M-1}(x)$ associated with the sequence $\mathcal{P}_0(x),\mathcal{P}_1(x),\ldots$ are the polynomials satisfying $\mathcal{Q}_0(x) = 1, \mathcal{Q}_1(x) = x - b_{M-1}$, and the recurrence
\[x\mathcal{Q}_i(x) = \mathcal{Q}_{i+1}(x) + b_{M-i-1}\mathcal{Q}_i(x) + u_{M-i}\mathcal{Q}_{i-1}(x)\]
for all $1 \le i \le M-2$.
\end{definition}
\begin{proposition} \cite{VZ} \label{prop:Q_recurrence} Let $M$ be a fixed positive integer. The finite sequence of dual polynomials $\mathcal{Q}_0(x),\ldots \mathcal{Q}_{M-1}(x)$ associated to $\mathcal{P}_0(x),\mathcal{P}_1(x),\ldots$ are orthogonal with respect to the discrete measure
\[
\sum_{i=1}^M w^*(x_i)\delta_{x_i},
\]
where
\[
w^*(x) = \frac{\mathcal{P}_{M-1}(x)}{\mathcal{P}_M(x)}
\]
and $x_1,\ldots, x_M$ are the roots of $\mathcal{P}_M(x)$. 
\end{proposition}
From now on we will assume that $Q_m^{(k)}(x)$ are the orthogonal polynomials corresponding to Laguerre polynomials via the relation $P_k(x) = (-1)^k k! L^{(N-k)}_k(x)$. The fact that $P_k(x)$ has multiple zero roots for $k > N$ adds a complication as the polynomial $\mathcal{P}_i(x)$ must have distinct roots. This is easily resolved by the relation \eqref{eq:laguerre_neg_alpha} between positive and negative $\alpha$ Laguerre polynomials. However, we are forced to consider the cases $k > N$ and $k \le N$ separately in the following results.
\begin{corollary} 
\label{cor:Q_recurrence}
For any fixed positive integer $k$, the polynomials of the form $Q_m^{(k)}(x)$ satisfy $Q_0^{(k)}(x) = 1, Q_1^{(k)}(x) = x - (k+N-1)$, and the recurrence
\begin{equation}
\label{eq:Q_recurrence}
xQ_{m}^{(k)}(x) = Q_{m+1}^{(k)}(x) + (N + k - 2m - 1)Q_{m}^{(k)}(x) + (k-m)(N-m)Q_{m-1}^{(k)}(x)
\end{equation}
whenever $1 \le m \le N \wedge k-1$.
\end{corollary}
\begin{proof}
Define $\mathcal{P}_i(x) = (-1)^ii!L^{(\alpha)}_i(x)$ for some fixed $\alpha$ and all nonnegative integers $i$. These satisfy a three-term recurrence of the form \eqref{eq:three_term_recurrence} with $b_i = 2i+1+\alpha, u_i = i(i+\alpha)$ and an additional recurrence
\[x\mathcal{P}_i'(x) = i\mathcal{P}_i(x) + i(i+\alpha)\mathcal{P}_{i-1}(x).\]
This latter recurrence yields
\[\frac{\mathcal{P}_{M-1}(x_i)}{\mathcal{P}_M'(x_i)} = \frac{x_i}{M(M+\alpha)}\]
for all roots $x_i$ of $\mathcal{P}_M(x)$. Finally, for $k > N$ take $M = N, \alpha = k - N$ and for $k \le N$ take $M = k$ and $\alpha = N - k$. Proposition \ref{prop:Q_recurrence} implies that $\mathcal{Q}_i(x) = Q_i^{(k)}(x)$ for $0 \le i \le M - 1$ as they have the same orthogonality measure, yielding the desired result.
\end{proof}
\begin{lemma} Let $k,l$ be positive integers satisfying $k \le l$ and suppose $0 \le m \le N \wedge k - 1$. Then
\begin{equation}
\label{eq:Q_norm}
\langle Q_m^{(k)}, Q_m^{(k)} \rangle_k = \frac{(k-m)_{m+1}(N-m)_{m+1}}{k+1},
\end{equation}
where $(c)_m = c(c+1)\cdots (c+m-1)$ is the Pochhammer symbol. Furthermore,
\begin{equation} 
\label{eq:Q_norm_ratio}
\frac{\langle Q_m^{(k)}, Q_m^{(k)} \rangle_{k}}{\langle Q_m^{(l)}, Q_m^{(l)} \rangle_l} = \frac{l+1}{k+1} \prod_{j=k}^{l-1}\left(1 - \frac{m+1}{j+1}\right).
\end{equation}
\end{lemma}
\begin{proof}
The idea for evaluating \eqref{eq:Q_norm} is to equate two different expressions for $\langle xQ_m^{(k)}, Q_{m-1}^{(k)} \rangle$. The first such expression arises from the inner product of \eqref{eq:Q_recurrence} with $Q_{m-1}^{(k)}$; the second comes from substituting $m \to m-1$ in \eqref{eq:Q_recurrence} and then taking the inner product with $Q_m^{(k)}$. One obtains
\[\langle Q^{(k)}_{m}, Q^{(k)}_{m} \rangle_{k}  = \langle xQ^{(k)}_m, Q^{(k)}_{m-1}\rangle  =  (k-m)(N-m)\langle Q^{(k)}_{m-1}, Q^{(k)}_{m-1} \rangle_{k}.\]
From here, \eqref{eq:Q_norm} follows by induction and \eqref{eq:Q_norm_ratio} is easily obtained by expanding out terms using \eqref{eq:Q_norm}.
\end{proof}
The final relation we will establish in this section is a contour integral expression for $Q_m^{(k)}$. For this we need a recurrence involving the polynomials $Q_m^{(k)}$ that holds for an infinite sequence rather than a finite one.
\begin{definition} (Associated Polynomials) Let $c$ be a nonnegative integer and let $\mathcal{P}_0(x), \mathcal{P}_1(x),\ldots$ be a sequence of polynomials satisfying the three-term recurrence \eqref{eq:three_term_recurrence}. The associated polynomials $\mathcal{P}_0(x;c), \mathcal{P}_1(x;c),\ldots$ corresponding to $\mathcal{P}_i(x)$ are the polynomials satisfying $\mathcal{P}_0(x;c) = 0, \mathcal{P}_{1}(x;c) = x - b_c$, and the recurrence
\[x\mathcal{P}_i(x;c) = \mathcal{P}_{i+1}(x;c) + b_{i + c}\mathcal{P}_i(x;c) + u_{i+c}\mathcal{P}_{i-1}(x;c)\]
for all $i \ge 1$.
\end{definition}
\begin{lemma} \cite{VZ}
\label{prop:dual_associated_relation}
For any nonnegative integers $i < M$, the dual and associated polynomials to $\mathcal{P}_0(x),\mathcal{P}_1(x),\ldots$ satisfy the equality
\[\mathcal{Q}_i(x) = \mathcal{P}_i(x; M - i).\]
\end{lemma}
\begin{proposition} \label{prop:Q_contour} Let $k$ be a positive integer suppose $0 \le m < N \wedge k$. The orthogonal polynomials $Q_m^{(k)}$ satisfy
\[Q_m^{(k)}(x) = \frac{(-1)^m(k \wedge N - m)_{m+1}}{2\pi i} \oint \int_0^t \frac{f(x,s)s^{-m}}{f(x,t)} \frac{dsdt}{s(1-s)t(1-t)},\]
where
\[f(x,s) = \exp\left(\frac{x}{1-s}\right)(1-s)^{|k - N|}s^{k \wedge N}\]
for any positively oriented contour containing 0 in its interior but not passing through or containing 1.
\end{proposition}
\begin{proof}
As before, let $\mathcal{P}_i(x) = (-1)^i i!L^{(\alpha)}_i(x)$ and let $c$ be a nonnegative integer. Define
\[g(x,t) = \sum_{i=0}^{\infty} \frac{\mathcal{P}_i(x;c)}{(c+1)_i}(-t)^i.\]
This is evaluated in \cite[(2.18)]{AW} by using the recurrence defining associated polynomials to establish a differential equation for $g(x,t)$. \cite{AW} showed
\[g(x,t) = \frac{c}{1-t}\int_0^t \exp\left(\frac{x}{1-s} - \frac{x}{1-t}\right) \left(\frac{1-s}{1-t}\right)^{\alpha} \left(\frac{s}{t}\right)^c \frac{ds}{s(1-s)}.\]
From here we get an expression for $\mathcal{P}_i(x;c)$ by integrating $g(x;t)t^{-i-1}$ along a contour containing 0 in its interior. Finally, the result follows from Lemma \ref{prop:dual_associated_relation} as the orthogonal polynomials $Q_m^{(k)}$ can be evaluated via associated polynomials. In the cases $k > N$, $k \le N$ one makes the same substitutions for $M,\alpha$ as in the proof of Corollary \ref{cor:Q_recurrence}. Lastly, we note that the contour cannot contain 1 as $f(x,t)$ has an essential singularity at $t = 1$.
\end{proof}
\begin{remark}
The essential singularity at 1 is the reason that this paper deals only with the hard edge limit instead both edge limits at zero temperature. In the soft edge limit, we could not find a way of adjusting the contours to make a steepest descent argument work on account of the essential singularity. This may seem a little strange as one can relate the Hermite polynomials to the Laguerre polynomials with parameter equal to $\alpha = \pm\frac{1}{2}$, and there were no corresponding issues with the Hermite analysis in \cite{GK}. The explanation for this is that there are actually two different integration identities at play here. For the Laguerre polynomials, the more analogous series to the one in \cite{GK} would be
\[g(x,t) = \sum_{i=0}^{\infty} \frac{\mathcal{P}_{i}(x;c)}{(c+1)_i(c+\alpha+1)_i} (-t)^{2i}.\]
Using the defining recurrence for associated polynomials one can show that this satisfies the differential equation:
\[\frac{g''(x,t)}{4} + \frac{4t^2 + 4c + 2\alpha + 1}{4t}g'(x,t) + (2c + \alpha + 1 + t^2 - x)g(x,t) + c(c+\alpha) \left(\frac{g(x,t)-1}{t^2}\right) = 0\]
where the derivatives are with respect to $t$. For $c > 0$, the only solution to this nonsingular at the origin is
\[g(x,t) = \pi c(c+\alpha) \int_0^t e^{s^2 - t^2} \left(\frac{s}{t}\right)^{\alpha + 2c}\left(Y_{\alpha}(-2it\sqrt{x})J_{\alpha}(-2is\sqrt{x}) - J_{\alpha}(-2it\sqrt{x})Y_{\alpha}(-2is\sqrt{x})\right)\frac{ds}{s},\]
where $J_{\alpha}, Y_{\alpha}$ are the Bessel functions of the first and second kind respectively. We remark that $g(x,t)$ is entire as a function of $t$ but in the general case $Y_{\alpha}$ needs a branch cut in order to be well-defined, so some care must be taken with this expression. In the Hermite case where $\alpha = \pm\frac{1}{2}$, $J_{\alpha}, Y_{\alpha}$ reduce to polynomial multiples of sine and cosine making this expression tractable. Unfortunately, for general $\alpha$, we were still unable to perform steepest descent using this formula on account of the difficulty in evaluating $Y_{\alpha}, J_{\alpha}$ for $\alpha$ growing polynomially in $N$, which would be neccessary for the soft edge case.
\end{remark}

\subsection{Covariances and Consequences} 
We will now decompose the summations from Proposition \ref{prop:laguerre_infinity_additive_polymer} in terms of the orthogonal polynomials $Q_m^{(k)}(x)$ to derive an alternate expression for $\cov(\xi_{a,k}, \xi_{a',k'})$. The tools we have established thus far are sufficient to decompose the diffusion kernels $K^{k,n}(a \to b)$. However, we still do not have the tools to do the same for the covariances of a single row. To this end, in the proof of Proposition \ref{prop:diagonalization_Q}, we will employ and more formally introduce the diagonalization techniques of \cite{AHV} touched upon in the introduction.
\begin{lemma} \label{lem:diffusive_kernel_expansion} The following holds for positive integers $r \ge k$:
\[
K^{k,r}(a \to b)  = \frac{l_{a,k}}{\sqrt{(k+1)(r+1)}}\sum_{m=0}^{k \wedge N - 1} \tilde{Q}_m^{(r)}(l_{b,r}) \tilde{Q}_m^{(k)}(l_{a,k}) \prod_{j=k}^{r-1} \left(1 - \frac{m+1}{j+1}\right)^{\frac12},
\]
where $a \le k \wedge N, b \le r \wedge N$.
\end{lemma}
\begin{proof}
Let $f(x) = \1(x = l_{a,k})$. Then
\begin{align*}
    K^{k,r}(a \to b) &= (D_{r-1}\cdots D_k f)(l_{b,r}) \\
    &= \sum_{m=0}^{k \wedge N - 1} (D_{r-1}\cdots D_k \tilde{Q}_m^{(k)})(l_{b,r}) \langle \tilde{Q}_m^{(k)}, f \rangle_k.
\end{align*}
Applying Proposition \ref{prop:diagonalization_Q}, the right-hand side becomes 
\[\sum_{m=0}^{k \wedge N - 1} \left( \tilde{Q}_m^{(r)}(l_{b,r}) \sqrt{\frac{\langle Q_m^{(r)}, Q_m^{(r)} \rangle_r}{\langle Q_m^{(k)}, Q_m^{(k)} \rangle_k}} \prod_{j=k}^{r-1} \left(1 - \frac{m+1}{j+1}\right) \right) \times \left(\frac{l_{a,k}}{k+1}\tilde{Q}_m^{(k)}(l_{a,k})\right), \]
which can be simplified with \eqref{eq:Q_norm_ratio} to get the desired.
\end{proof}
\begin{proposition}
\label{prop:covariance_summation}
Let $\{\xi_{a,k}\}_{1 \le a \le k \le n}$ be random variables distributed according to the Laguerre $\infty$-corners process with $n$ rows. Then for positive integers $k \ge k'$,
\begin{align*} \cov(\xi_{a,k}, \xi_{a',k'}) =& \frac{2l_{a,k}l_{a',k'}}{n+1}\sum_{m=0}^{k' \wedge N -1} \frac{\tilde{Q}_{m}^{(k)}(l_{a,k})\tilde{Q}_{m}^{(k')}(l_{a',k'})}{(m+1)\sqrt{(k+1)(k'+1)}}\prod_{j=k}^{n-1} \left(1 - \frac{m+1}{j+1}\right)^{\frac12} \prod_{j'=k'}^{n-1} \left(1 - \frac{m+1}{j'+1}\right)^{\frac12} \\ &+ \sum_{r=k}^{n-1} \frac{2 l_{a,k} l_{a',k'}}{r+1} \sum_{m=0}^{k' \wedge N-1}  \frac{\tilde{Q}_{m}^{(k)}(l_{a,k})\tilde{Q}_{m}^{(k)}(l_{a',k'})}{\sqrt{(k+1)(k'+1)}}\prod_{j=k}^{r-1} \left(1 - \frac{m+1}{j+1}\right)^{\frac12} \prod_{j'=k'}^{r-1} \left(1 - \frac{m+1}{j'+1}\right)^{\frac12}.
\end{align*}
\end{proposition}
\begin{proof}
Using Proposition \ref{prop:laguerre_infinity_additive_polymer},
\begin{align} 
\cov(\xi_{a,k}, \xi_{a', k'}) =& \sum_{b,b'=1}^n K^{k,n}(a \to b)K^{k',n}(a' \to b')\cov(\xi_{b,n}, \xi_{b',n})  \label{eq:covar_decomp_diffusive_kernel} \\
&+ \sum_{r=k}^{n-1} \sum_{b=1}^r K^{k,r}(a \to b)K^{k',r}(a' \to b) \var(\eta_{b,r}).  \nonumber
\end{align}
These two summations will respectively equal the two from the proposition statement. By Lemma \ref{lem:diffusive_kernel_expansion}, the latter one equals
\begin{align} \sum_{r=k}^{n-1}\sum_{b=1}^r \sum_{m_1 = 0}^{k \wedge N-1} \sum_{m_2 = 0}^{k' \wedge N - 1} &\frac{2l_{a,k}l_{a',k'}l_{b,r}}{(r+1)^2\sqrt{(k'+1)(k+1)}} \tilde{Q}_{m_1}^{(k)}(l_{a,k})\tilde{Q}_{m_1}^{(r)}(l_{b,r})\tilde{Q}_{m_2}^{(k)}(l_{a',k'})\tilde{Q}_{m_2}^{(r)}(l_{b,r}) \times  \label{eq:polymer_second_term_expansion} \\
&\times \prod_{j=k}^{r-1} \left(1 - \frac{m_1+1}{j+1}\right)^{\frac12} \prod_{j'=k'}^{r-1} \left(1 - \frac{m_2+1}{j'+1}\right)^{\frac12}. \nonumber 
\end{align}
Recall by orthogonality,
\begin{equation} \label{eq:Q_orthogonality_relation}
\sum_{b=1}^r \frac{l_{b,r}}{r+1} \tilde{Q}_{m_1}^{(r)}(l_{b,r})\tilde{Q}_{m_2}^{(r)}(l_{b,r}) = \1(m_1 = m_2).
\end{equation}
Applying this relation to the $b$-dependent terms of \eqref{eq:polymer_second_term_expansion} gives the desired summation.
\par 
For the first summation of \eqref{eq:covar_decomp_diffusive_kernel}, \cite{AHV} showed that the matrix
\[
\Sigma = \left\{\frac{1}{4l_{b,n}l_{b',n}}\cov\left(\xi_{b,n},\xi_{b',n}\right)\right\}_{1 \le b,b' \le N}
\]
has an orthonormal basis of eigenvectors given by
\[
v_m = \frac{1}{\sqrt{n+1}} \langle l_{1,n}^{1/2}\tilde{Q}_m^{(n)}(l_{1,n}),\ldots l_{N,n}^{1/2} \tilde{Q}_{m}^{(n)}(l_{N,n}) \rangle
\]
for $m = 0,1,\ldots, N - 1$ with corresponding eigenvalues $\frac{1}{2}, \frac{1}{4},\ldots, \frac{1}{2N}$. Hence,
\[
\cov(\xi_{b,n}, \xi_{b',n}) = \frac{2l_{b,n}l_{b',n}}{n+1} \sum_{m=0}^{N-1} \frac{1}{m+1} \tilde{Q}_m^{(n)}(l_{b,n})\tilde{Q}_m^{(n)}(l_{b',n}).
\]
With this equality along with Lemma \ref{lem:diffusive_kernel_expansion}, the first summation of \eqref{eq:covar_decomp_diffusive_kernel} becomes
\begin{align*} 
\sum_{b,b' = 1}^n \frac{2l_{b,n}l_{b',n}l_{a,k}l_{a',k'}}{(n+1)^2}\sum_{m_1=0}^{k-1}\sum_{m_2 = 0}^{k'-1}\sum_{m_3 = 0}^{N-1} &\frac{\tilde{Q}_{m_1}^{(n)}(l_{b,n})\tilde{Q}_{m_1}^{(k)}(l_{a,k})\tilde{Q}_{m_2}^{(n)}(l_{b',n})\tilde{Q}_{m_2}^{(k')}(l_{a',k'})}{\sqrt{(k+1)(k'+1)}}\frac{\tilde{Q}_{m_3}^{(n)}(l_{b,n})\tilde{Q}_{m_3}^{(n)}(l_{b',n})}{m_3+1} \times \\
&\times \prod_{j=k}^{n-1} \left(1 - \frac{m_1+1}{j+1}\right)^{\frac12} \prod_{j'=k'}^{n-1} \left(1 - \frac{m_2+1}{j'+1}\right)^{\frac12}.
\end{align*}
Now, interchange the summations such that the ones over $b,b'$ are the innermost ones. One sees that both summations are of the same form as the left-hand side of the orthogonality relation \eqref{eq:Q_orthogonality_relation}. Simplifying these in a similar manner yields the desired.
\end{proof}

\section{Hard Edge Limit} 
\label{sec:hard_edge_limit}
This section will be devoted to proving our main and secondary results. Our proofs will begin in Section \ref{sec:hard_edge_Q_asymp}, where we apply steepest descent to our contour integral formula in order to compute the asymptotics of the orthogonal polynomials $Q_m^{(k)}$ evaluated as the smallest roots of Laguerre polynomials. Following this, in Section \ref{sec:proof_main_hard_edge}, we apply these asymptotics to Proposition \ref{prop:covariance_summation} to obtain the main result.
\par 
Before we proceed with the proofs, however, we begin with a brief exposition on Sturm-Liouville theory and its connection to Bessel and Airy functions. The goal of this exposition will be to give some informal intuition as to why our hard edge result is a natural counterpart to the soft edge results of \cite{GK, AHV}. This discussion will also prove useful in our final Section \ref{sec:additive_polymer_limit}, which will be devoted to proving Theorem \ref{thm:secondary_polymer_result}.

\subsection{Airy, Bessel Bases through Sturm-Liouville Theory} 
\label{sec:airy_bessel_sturm}
The Airy function $\Ai(x)$ is the solution to the differential equation
\[y'' = xy\] 
satisfying $\Ai(x) \to 0$ as $x \to \infty$. It was shown in \cite{GK, AHV} that when our Appell sequence $P_0(x),P_1(x),\ldots$ is given by the Hermite polynomials, under proper scaling the asymptotics of $Q_m^{(k)}(h_{i,k})$ as $k \to \infty$ is given by
\begin{equation} \label{eq:Fourier_Airy}
\frac{\Ai(\mathfrak{a}_i + x)}{\Ai' (\mathfrak{a}_i)}
\end{equation}
where $i$ is some fixed positive integer, $x$ is some scaling of $m$, $h_{i,k}$ is the $i$-th smallest root of the Hermite polynomial $H_k$, and $\mathfrak{a}_1 > \mathfrak{a}_2 > \cdots$ are the real roots of the Airy function. Now, consider the matrix with rows given by $v_m$ as defined in the proof of Proposition \ref{prop:covariance_summation}. This matrix is unitary and thus we get orthogonality of columns:
\begin{equation} \label{eq:column_orthogonality}
\sum_{m=0}^{N-1} \tilde{Q}^{(n)}_m(l_{i,n})\tilde{Q}^{(n)}_m(l_{j,n}) = \frac{n+1}{l_{i,n}}\1(i = j).
\end{equation}
for all positive integers $i,j \le n$. In the Hermite case, the same identity holds with a different constant on the right hand side and $n = N$. One might then expect that the limiting functions \eqref{eq:Fourier_Airy} are orthonormal in $L^2((0,\infty))$. Not only does this orthogonality relation hold true, but they also form a basis for a large class of functions. Indeed, the limiting functions are referred to as the Fourier-Airy series and this orthonormality results falls under the larger framework of Sturm-Liouville theory (see \cite{Tit} for the classic primer), which deals with the eigenvalue problem $Ly = \lambda y$ for operators of the form
\[L = \frac{1}{\omega(x)} \left(-\frac{d}{dx} \left(p(x)\frac{d}{dx}\right) + q(x)\right).\]
Suppose $-\infty \le a < b \le \infty$. Under certain regularity constraints on the functions $p(x), q(x), \omega(x)$ and boundary conditions at $a,b$, the eigenfunctions $y_n$ of $L$ forms an orthonormal basis with respect to the weight function $w(x)$. The Fourier-Airy series consists of precisely these eigenfunctions when $L$ corresponds to the Airy differential equation.
\par 
As one might expect from the finite $\beta$ case, the Bessel function emerges in place of the Airy function in the hard edge limit. The Bessel function of the first kind $J_{v}(x)$ is defined to be the solution to the differential equation 
\begin{equation} \label{eqn:Bessel_DE}
x^2y'' + xy' + (x^2 - v^2)y = 0
\end{equation}
that satisfies
\[J_{v}(x) = \sum_{k=0}^{\infty} (-1)^k \frac{(x/2)^{2k+v}}{k!\Gamma(k+v+1)}.\] As we will see regime of interest for the hard edge involves a spatial scaling of $N^{-1}$. Under this scaling, as $N \to \infty$ \eqref{eq:column_orthogonality} becomes an integral with orthogonality weight 1 over $(0,1)$ indicating $a = 0, b = 1$. The solutions to the eigenvalue problem for $L$ corresponding to the Bessel differential equation ($w(x) = p(x) = x$ and $q(x) = \frac{v^2}{x}$) are given by
\begin{equation} \label{eq:Bessel_Fourier}
\frac{J_v(j_{i,v}y)}{J_v'(j_{i,v})}.
\end{equation}
As will be shown in the next section, the limit of $Q_m^{(k)}(l_{k+1-i,k})$ is indeed given by these eigenfunctions transformed to have orthogonality weight 1, which is reasonable in light of \eqref{eq:column_orthogonality}
\par 
The functions \eqref{eq:Bessel_Fourier} are referred to as the Fourier-Bessel series. We end with a definition and a lemma that states conditions under which a function can be decomposed pointwise in terms of the Fourier-Bessel series. This lemma will also prove useful in Section \ref{sec:additive_polymer_limit}. 
\begin{definition} (Normalized Fourier-Bessel Series) In analogy with the orthogonal polynomials $\tilde{Q}_m^{(k)}(x)$ define
\[\tilde{J}_{b,v}(y) = \frac{J_v(j_{b,v}y)}{J_v'(j_{b,v})}.\]
\end{definition}
\begin{lemma}\cite[Sec. 18.24]{Wat}\label{lem:Tit46_laguerre_expansion} For positive integers $i \neq j$ and nonnegative integers $v$, the functions $\tilde{J}_{i,v}, \tilde{J}_{j,v}$ are orthogonal in $L^2((0,1), w(y) = y)$. Furthermore, if $f:[0,1]\to\R$ is a continuous function having finite total variation in the interval $(a,b)$ for some $0 < a < b < 1$, then the following Fourier-like decomposition holds for all $y \in (a,b)$:
\[f(y) = 2 \sum_{i = 0}^{\infty}\tilde{J}_{i,v}(y) \int_0^1 \tilde{J}_{i,v}(z)f(z)zdz. \]
\end{lemma} 

\subsection{Asymptotics of Orthogonal Polynomials at the Hard Edge}
This section will be devoted to the proof of Proposition \ref{prop:Q_norm_hard_edge}, which will give us the necessary tools to compute $Q_m^{(k)}$ evaluated at the smallest roots of the Laguerre polynomials as $N \to \infty$. As will be demonstrated, one of the steps in this proof is to replace the root $l_{k+1-r,k}$ of a Laguerre function by the corresponding root of a Bessel function that describes its asymptotics. To ensure this correspondence holds regardless of the sign of $k-N$ define
\[j_{r+v,-v} = j_{r,v}\]
whenever $v, r > 0$. Indeed, $J_{-v}(x) = (-1)^v J_{v}(x)$ so the values $j_{r+v,-v}$ will all be roots of $J_{-v}(x)$. We also set $j_{r+v,-v} = 0$ if $0 \ge r > -v$ as the corresponding Laguerre polynomials will have $v$ roots equal to 0.
\label{sec:hard_edge_Q_asymp}.
\begin{proposition} \label{prop:Q_norm_hard_edge} Suppose $k,N$ tend to $\infty$ in such a way that $N-k$ equals a fixed constant $\alpha$. For any fixed positive integer $r > k - N$, one has
\begin{equation} \label{eq:asymptotic_Q_formula_hard_edge}
\frac{(-1)^m\tilde{Q}_m^{(k)}(l_{k+1 - r, k})}{N^{1/2}} = \frac{2}{j_{r,\alpha}} \cdot \frac{J_{\alpha}(j_{r,\alpha}\sqrt{1-\frac{m}{k}})}{J_{\alpha}'(j_{r,\alpha})} + O(N^{-1}),
\end{equation}
uniformly over all $0 \le m < N \wedge k$.
\end{proposition}
\begin{proof}
The cases for $k \le N$ and $k > N$ are nearly identical, so we will proceed by first giving an in depth proof when $k \le N$ and then briefly remarking on what changes one needs to make for the other case. Assuming $k\le N$, \eqref{eq:Q_norm} and Proposition \ref{prop:Q_contour} together imply
\begin{equation} \label{eq:Q_normalized_asymp_1}
(-1)^m\tilde{Q}_m^{(k)}(x) =  -\frac{\sqrt{k+1}}{2\pi i}\sqrt{\frac{(k-m)_{m+1}}{(N-m)_{m+1}}} \oint \int_0^t \frac{f(x,s)s^{-m}}{f(x,t)t} \frac{dsdt}{s(1-s)t(1-t)},
\end{equation}
where
\[f(x,s) = \exp\left(\frac{xs}{1-s}\right)(1-s)^{N-k}s^k\]
and the contour is a positively oriented curve containing 0 but not 1. The constant in front of the integral evaluates as
\begin{equation} \label{eq:hard_edge_product_term}
\sqrt{\frac{(k-m)_{m+1}}{(N-m)_{m+1}}} = \prod_{j=k}^{N-1} \left(1 - \frac{m+1}{j+1}\right)^{1/2} = (1-y)^{\alpha/2} + O(N^{-1}),
\end{equation}
where $y = \frac{m}{k}$. We will now focus our attention on the integral itself. Let $\gamma:[0,1] \to \C$ be a parametrization of the outer contour. Integrating by parts, we get
\begin{align}  \label{eq:airy_hard_edge_integration_by_parts}
\oint_{\gamma} \int_0^t \frac{f(x,s)s^{-m}}{f(x,t)t} \frac{dsdt}{s(1-s)t(1-t)} =& \oint_{\gamma} \frac{f(x,s)}{s^{m+1}(1-s)}ds \oint_{\gamma} \frac{1}{f(x,t)t(1-t)}dt \nonumber \\
&- \oint_{\gamma} \int_{\gamma(0)}^{t} \frac{f(x,t)t^{-m}}{f(x,s)s} \frac{dsdt}{s(1-s)t(1-t)},
\end{align}
where the inner integral of the second term runs along $\gamma$ from $\gamma(0)$ to $t$. As $\frac{f(x,s)}{s^{m+1}(1-s)}$ has no poles inside $\gamma$, the first term vanishes. Now assume $x = l_{k+1-r,k}$. The standard contour integral identity for the Laguerre polynomials yields
\[
\oint_{\gamma} \frac{du}{f(l_{k+1-r,k},s)(1-u)u} = L^{(N-k)}_k(l_{k+1-r,k}) = 0.
\]
This identity implies that for the inner contour of the final term in \eqref{eq:airy_hard_edge_integration_by_parts}, we can travel along either the positive or negative direction of $\gamma$. This will be important when selecting the curve $\gamma$.
\par 
For fixed $\nu$, the Laguerre polynomials satisfy the following asymptotic relation as $n\to\infty$ \cite[p. 199]{Sze}: 
\[e^{-x/2}x^{\nu/2}L_n^{(\nu)}(x) = M^{-\nu/2}\frac{\Gamma(n+\nu + 1)}{\Gamma(n+1)}J_{\nu}(2(Mx)^{1/2})+ O(n^{\nu/2-3/4}),\]
where $M = n + \frac{\nu+1}{2}$ and for any fixed constant $w$ the bound holds uniformly for $0 < x < w$. As all roots of $L^{\nu}_n(x)$ are positive, one obtains by way of Rouche's theorem the following correspondence between roots of Laguerre polynomials and Bessel functions:
\begin{equation} \label{eq:lag_root_hard}
l_{k+1-r,k} = \frac{j_{r,\alpha}^2}{4N} + O(N^{-7/4}).
\end{equation}
Expanding out $f(l_{k+1-r,k},s)$ using this asymptotic formula, we get that the integrand of \eqref{eq:airy_hard_edge_integration_by_parts} equals
\begin{equation} \label{eq:hard_edge_Q_full_integrand} \frac{\exp\left(k\ln\left(\frac{t}{s}\right) + \alpha\ln\left(\frac{1-t}{1-s}\right) + \frac{j_{r,\alpha}^2}{4N}\left(\frac{t}{1-t} - \frac{s}{1-s}\right) - m\log(-t) + O(N^{-7/4})\right)}{s(1-s)t(1-t)}.
\end{equation} 
Our approach is steepest descent, so we would like to select a contour for which this decays exponentially in $N$ outside of a small portion of $\gamma$. To this end, select $\gamma$ to be a curve which first travels from -2 to $1 - N^{-1} + e^{\frac{-\pi i}{4}}$ with $|t|$ strictly decreasing, then travels from $1 - N^{-1} + e^{\frac{-\pi i}{4}}$ to $1 - N^{-1}$ in a straight line, and lastly follows the reflection of the first two segments over the real line back to -2. Per our earlier discussion, the inner contour will travel along the shorter arc of $\gamma$ from -2 to $t$.
\par 
Indeed, this choice of contour ensures that outside of an $O(N^{-1})$ ball around 1, we have $1 < |t| < |s|$ and thus the real parts of $k\ln\left(\frac{t}{s}\right), -m \log(-t)$ are negative. The only remaining term non-negligible outside of the $N^{-1}$ scale is $\alpha \ln\left(\frac{1-t}{1-s}\right)$. However, this is dominated by the other logarithmic term at the constant scale and has negative real part at any smaller scale as the curve satisfies $|1-t| < |1-s|$. Hence, all asymptotically non-vanishing terms of \eqref{eq:hard_edge_Q_full_integrand} will be captured by a Taylor expansion about 1. As such, define
\[t = 1 - \frac{v}{N}, \qquad s = 1 - \frac{u}{N}.\]
Asymptotically, \eqref{eq:airy_hard_edge_integration_by_parts} Taylor expands out to
\begin{equation} 
\label{eq:hard_edge_Q_integral_asymptotic_formula}
G(y) := \int_{1+e^{-\frac{ 3\pi i}{4}}\infty}^{1+e^{\frac{3 \pi i}{4}}\infty} \int_{1 + e^{\pm \frac{3 \pi i}{4}}\infty}^{v} \left( \frac{\exp(F(v) - F(u) + yv)}{uv}\right)du dv,
\end{equation}
where
\[F(x) = \frac{j_{r, \alpha}^2}{4x} - x + \alpha \ln(x).\]
The notation used for the lower bound of the inner integral is shorthand indicating that the contour approaches $v$ from $1 + e^{ -\frac{3 \pi i}{4}}\infty$ if $\text{Im}(v) < 0$ and $1 + e^{\frac{3 \pi i}{4}}\infty$ otherwise. We also note that as our Taylor expansion is of order $N^{-1}$, $G(y)$ will differ from \eqref{eq:airy_hard_edge_integration_by_parts} by a factor of the same order.
\par 
It is not immediately clear that $\tilde{Q}_m^{(k)}(l_{k+1-r,k})$ has a limiting value, so it still remains to see that $G(y)$ is well-defined. Integrating by parts,
\[\int_{1+e^{-\frac{3\pi i}{4}}}^v \frac{\exp(-F(u))}{u}du = \frac{v\exp(-F(v))}{v-v^2F'(v)} - \int_{1+e^{-\frac{3\pi i}{4}}}^v \frac{d}{du}\left[\frac{1}{u-u^2F'(u)}\right]u\exp(-F(u))du.\]
As the integration contour stays away from 0, both terms on the right-hand side are bounded in magnitude by a factor of order $|v|^{-\alpha - 1}$. Hence, $G(y)$ is bounded by a constant multiple of 
\[\int_{1+e^{\frac{3\pi i}{4}}\infty}^{1} \left| \frac{\exp(yv)}{v^2} \right|dv + \int_{1+e^{-\frac{3\pi i}{4}}\infty}^{1} \left| \frac{\exp(yv)}{v^2} \right|dv.\]
Consider the open cone
\[ \mathcal{C} = \left\{z \ : \ \text{arg}(z) \in \left(-\frac{\pi}{4}, \frac{\pi}{4}\right)\right\}.\]
If $y \in \mathcal{C}$, then as the $v$ moves away from $1$ along the contour $\exp(yv)$ decays exponentially. As such, over $\mathcal{C}$, $G(y)$ is holomorphic, non-singular, and we can differentiate it under the integral sign. For $y$ on the boundary of $\mathcal{C}$, $\exp(yv)$ remains bounded in magnitude so $G(y)$ will remain well-defined and thus exists for all $y \in [0,1]$. These properties will be the key to identifying it with desired expression. 
\par 
Indeed, define the differential operator
\begin{equation} \label{eqn:Operator_D}
\mathcal{D} = (1-y)^2 \frac{\partial^2}{\partial y^2} - (1-y)\frac{\partial}{\partial y} + \left(\frac{j_{r,\alpha}^2}{4}(1-y)-\frac{\alpha^2}{4}\right).
\end{equation}
Applying $\mathcal{D}$ to $(1-y)^{\alpha/2}G(y)$ yields
\[
(1-y)^{\frac{\alpha}{2}+1}\int_{1+e^{-\frac{ 3\pi i}{4}}\infty}^{1+e^{\frac{3 \pi i}{4}}\infty} \int_{1 + e^{\pm \frac{3 \pi i}{4}}\infty}^{v}  \frac{\exp(F(v) - F(u) + yv)}{uv} \left(\frac{j_{r,\alpha}^2}{4} - (\alpha + 1)v + (1-y)v\right)du dv,
\]
which may be more succinctly written as
\[
(1-y)^{\frac{\alpha}{2}+1}\int_{1+e^{-\frac{ 3\pi i}{4}}\infty}^{1+e^{\frac{3 \pi i}{4}}\infty} \left(\int_{1 + e^{\pm \frac{3 \pi i}{4}}\infty}^{v}  \frac{\exp(F(u))}{u}du\right)\frac{\partial}{\partial v} \left(v\exp(F(v) + yv)\right) dv.
\]
Using integration by parts to move the derivative in front of the inner integral yields
\[(1-y)^{\frac{\alpha}{2}+1}\int_{1+e^{-\frac{ 3\pi i}{4}}\infty}^{1+e^{\frac{3 \pi i}{4}}\infty} \exp(yv) dv,\]
which vanishes. The operator $\mathcal{D}$ can be identified with a transformation of the Bessel differential operator \eqref{eqn:Bessel_DE}. More exactly, the vanishing of $\mathcal{D}(1-y)^{\alpha/2}G(y)$ implies there exists a $g(x)$ satisfying \eqref{eqn:Bessel_DE} with
\begin{equation} \label{eqn:kernel_D}
(1-y)^{\alpha/2}G(y) = g(j_{r,\alpha}\sqrt{1-y}). 
\end{equation}
As the left-hand side is additionally non-singular at $y = 1$, the function $g(x)$ must be a constant multiple of $J_{\alpha}(x)$. 
\par 
In total, we have now shown
\begin{align} 
(-1)^mN^{-1/2}\tilde{Q}_m^{(k)}(x)
&= cJ_{\alpha}(j_{r,\alpha}\sqrt{1-y}) + O(N^{-1}). \label{eq:hard_edge_final_eq_wo_constant}
\end{align}
All that remains is to evaluate the constant $c$. Since \eqref{eq:hard_edge_final_eq_wo_constant} holds uniformly for $y\in [0,1]$, the orthogonality of columns relation \eqref{eq:column_orthogonality} becomes an integral relation in the limit:
\[c^2 \int_0^1 J_{\alpha}(j_{r,\alpha}\sqrt{1-y})^2 = \frac{4}{j_{r,\alpha}^2}.\]
Setting $f(y) = \tilde{J}_{i,v}(z)$ in Lemma \ref{lem:Tit46_laguerre_expansion} gives the equality
\[2\int_0^1  \frac{J_{v}(j_{i,v}z)^2}{J_v'(j_{i,v})^2}zdz = 1.\]
From here, a simple transformation of variables allows us to compute $c$ to be the desired value.
\par 
This completes the proof of \eqref{eq:asymptotic_Q_formula_hard_edge} for $k \le N$. For $k > N$ we instead get
\[ (-1)^m\tilde{Q}_m^{(k)}(x) =  -\frac{\sqrt{k+1}}{2\pi i}\sqrt{\frac{(N-m)_{m+1}}{(k-m)_{m+1}}} \oint \int_0^t \frac{f(x,s)s^{-m}}{f(x,t)t} \frac{dsdt}{s(1-s)t(1-t)} \]
where
\[f(x,s) = \exp\left(\frac{xs}{1-s}\right)(1-s)^{k-N}s^N.\]
As $k-N$ is a non-negative constant, $f(x,s)$ still has poles at the relevant points and our computations involving it remain unaffected. Similarly, we have  
\[\sqrt{\frac{(N-m)_{m+1}}{(k-m)_{m+1}}} = (1-y)^{-\alpha/2} + O(N^{-1})\]
preserving \eqref{eq:hard_edge_product_term} and our reasoning as $-\alpha > 0$. Finally, on account of the formula \eqref{eq:laguerre_neg_alpha} for $L_{N}^{(\alpha)}(x)$, the asymptotics of $l_{k+1-r,k}$ are still given by \eqref{eq:lag_root_hard} under the sustitution $-\alpha \to \alpha$ . As these exhaust the differences that arise, our argument still holds for $k > N$.
\end{proof}

\subsection{Proof of the Main Result}
\label{sec:proof_main_hard_edge}
\begin{proof}[Proof of Theorem \ref{thm:main_hard_edge}]
The method here will simply be to apply our asymptotic results to the orthogonal polynomial decomposition of the covariances in Proposition \ref{prop:covariance_summation}. We begin by dealing with the latter summation from the proposition. To shorten our expressions we also set $k = N + s, k' = N + t$ and assume $s \ge t$. Using \eqref{eq:hard_edge_product_term} we can first simplify the product terms of the summand to get
\begin{align*}
\sum_{r=k}^{n-1} \frac{1}{r+1} \sum_{m=0}^{k'\wedge N - 1} \frac{\tilde{Q}_m^{(k)}(l_{k+1-a,k})\tilde{Q}_m^{(k')}(l_{k'+1-b,k'})}{\sqrt{(k+1)(k'+1)}}(1-y)^{r-\frac{k+k'}{2}} + O(N^{-1}),
\end{align*}
where $y = \frac{m+1}{N}$. The only $r$ dependent terms are $\frac{1}{r+1}$, which is asymptotically $\frac{1}{N}$, and $(1-y)^{r-\frac{k+k'}{2}}$. Interchanging the summations and simplifying thus yields
\[O(N^{-1}) + \frac{1}{N} \sum_{m=0}^{k'\wedge N - 1} \frac{\tilde{Q}_m^{(k)}(l_{k+1-a,k})\tilde{Q}_m^{(k')}(l_{k'+1-b,k'})}{\sqrt{(k+1)(k'+1)}} \left(\frac{(1-y)^{|s-t|/2}-(1-y)^{n-\frac{k_1+k_2}{2}}}{y}\right).\]
Analogous reasoning applied to the first summation of Proposition \ref{prop:covariance_summation} gives
\[ O(N^{-1}) + \frac{1}{N} \sum_{m=0}^{k'\wedge N - 1} \frac{\tilde{Q}_m^{(k)}(l_{k+1-a,k})\tilde{Q}_m^{(k')}(l_{k'+1-b,k'})}{\sqrt{(k+1)(k'+1)}} \left(\frac{(1-y)^{n-\frac{k_1+k_2}{2}}}{y}\right). \]
Thus altogether we get
\[\Xi(a,s,b,t) = 2l_{k+1-a,k}l_{k'+1-b,k'}\left(O(N^{-1}) + \frac{1}{N} \sum_{m=0}^{k'\wedge N - 1} \frac{\tilde{Q}_m^{(k)}(l_{k+1-a,k})\tilde{Q}_m^{(k')}(l_{k'+1-b,k'})}{\sqrt{(k+1)(k'+1)}} \frac{(1-y)^{|s-t|/2}}{y}\right).\]
Finally, replacing $l_{k+1-a,k}$ and the orthogonal polynomials with their respective asymptotic expressions \eqref{eq:lag_root_hard}, \eqref{eq:asymptotic_Q_formula_hard_edge} gives the Riemann sum
\[\frac{j_{a,s}j_{b,t}}{2N^3} \sum_{m=0}^{k'\wedge N - 1} \left(\frac{J_{s}(j_{a,s}\sqrt{1-y})J_{t}(j_{b,t}\sqrt{1-y})}{J_{s}'(j_{a,s})J_{t}'(j_{b,t})} + O(N^{-1})\right) \frac{(1-y)^{|s-t|/2}}{y}.\]
which is easily identified with the desired expression.
\end{proof}

\subsection{Additive Polymer Limit}
\label{sec:additive_polymer_limit}
This section will be devoted to the proof of our secondary result. As explained in the introduction, this result demonstrates that the limiting Gaussians from Theorem \ref{thm:main_hard_edge} are equal to the partition function of a certain additive polymer arising from the Bessel-$\infty$ random walk. However, there is alternative interpretation of this result. To describe this, recall how the argument used in this paper proceeded: first we demonstrated that for all relevant $a,k$ the random variable $\xi_{a,k}$ equals the partition function of an additive polymer arising from a random walk on the roots of Laguerre polynomials; following that, we used a diagonalization of the random walk to express the covariances $\cov(\xi_{a,k}, \xi_{b,k'})$ in terms of orthogonal polynomials; and lastly, we used asymptotic techniques to take the $N \to \infty$ limit of said expression. We will prove that the additive polymer described by Theorem \ref{thm:secondary_polymer_result} arises by taking the large $N$ limit earlier in this sequence. In particular, as $N\to\infty$ the random walk on the roots of the Laguerre converges to the Bessel-$\infty$ random walk.

\begin{definition} (Bessel-$\infty$ Random Walk) Let $S_v = \{j_{i,v}\}_{i=1}^{\infty}$ where $v$ is some integer. The Bessel-$\infty$ Random Walk is a Markov chain with state space at time $t$ given by $S_{\alpha - t}$ for some $\alpha$ and transition probability from $j_{a,v}$ to $j_{b,v-1}$ given by
\[P^{v,v-1}(a \to b) = \frac{4j_{a,v}^2}{(j_{a,v}^2 - j_{b,v-1}^2)^2}.\]
In the case that $j_{a,v} = 0$, we instead transition to one of the 0 values of $S_{v-1}$ with uniform probability. For $v_2 < v_1$, we also define $P^{v_1, v_2}$ to be given by composing the transition matrices $P^{v_1,v_1-1}, \cdots, P^{v_2+1,v_2}$.
\end{definition}

Applying the asymptotic formula \eqref{eq:lag_root_hard} for $l_{k+1-i,k}$ to our formulas \eqref{eq:weight_evaluation}, \eqref{eq:laguerre_trans_prob_one_level} for the transition probabilities $\alpha_{a,b}^k$ yields
\begin{equation} \label{eq:diffusive_kernel_limit}
P^{v, v-1}(a \to b) = \lim_{N\to\infty} K^{N-v, N-v+1 }(N+1-v-a \to N+2-v-b). 
\end{equation}
Thus the single level transition probabilities for the Bessel-$\infty$ process are the exactly limit of the corresponding probabilities for the jumping process of Section \ref{sec:additive_polymer}. To show that this relation can be extended to $P^{v_1,v_2}$ for any $v_1, v_2$ rather than just $v_2 = v_1 - 1$, we need the following lemma demonstrating that the Bessel-$\infty$ process is indeed a well-defined Markov chain. A similar result relating transition probabilities from $S_v$ to itself can be found in \cite{Cal}.
\begin{lemma} For all positive integers $a$ and integers $v$ with $j_{a,v} \neq 0$ the following holds:
\[\sum_{b=1}^{\infty} \frac{4j_{a,v}^2}{(j_{a,v}^2 - j_{b,v-1}^2)^2} = 1.\]
\end{lemma}
\begin{proof} Suppose $v \ge 0$. By \cite{Wat} Section 15.41, the following Mittag-Leffler expansion holds:
\[-\frac{J_{v}(x)}{J_{v-1}(x)} = \sum_{r=1}^{\infty} \frac{2x}{x^2 - j_{r,v-1}^2}.\]
Taking the derivative with respect to $x$,
\[\frac{J_{v-1}'(x)J_{v}(x) - J_{v}'(x)J_{v-1}(x)}{J_{v-1}(x)^2} = \sum_{r=1}^{\infty} \left(\frac{2}{x^2 - j_{r,v-1}^2} - \frac{4x^2}{(x^2 - j_{r,v-1}^2)^2}\right).\]
Plugging $x=j_{a,v}$ into both equations and using the identity $J_{v}'(j_{a,v}) = J_{v-1}(j_{a,v})$ gives the desired. For $v < 0$, we need to make a small adjustment on account of the $r$ for which $j_{r,-v} = 0$. Applying
\[2vJ_v(x) = x(J_{v-1}(x) + J_{v+1}(x))\]
to the Mittag-Leffler expansion yields
\[\frac{J_{v-1}(x)}{J_v(x)} = \sum_{r=1}^{\infty} \frac{2x}{x^2 - j_{r,-v}^2}.\]
From here we can proceed as before.
\end{proof}

\begin{proposition} \label{prop:additive_polymer_transition_probs} For positive integers $a,b$ and integers $v_1 \ge v_2$ with $j_{b,v_2} \neq 0$, 
\begin{equation} \label{eq:additive_polymer_transition_probs}
P^{v_1,v_2}(a \to b) = \frac{j_{a,v_1}}{j_{b,v_2}}\int_0^1\tilde{J}_{a,v_1}(\sqrt{1-y})\tilde{J}_{b,v_2}(\sqrt{1-y})(1-y)^{|v_1 - v_2|/2}dy,
\end{equation}
and thus one has 
\begin{equation} \label{eq:general_diffusion_kernel_limit}
P^{v_1, v_2}(a \to b) = \lim_{N\to\infty} K^{N-v_1, N-v_2}(N+1-v_1-a \to N+1-v_2-b).
\end{equation}
\end{proposition}
\begin{proof}
Let $R^{v_1,v_2}(a \to b)$ be equal to the right-hand side of \eqref{eq:additive_polymer_transition_probs}. For brevity's sake also set $n_1 = N - v$ and $n_2 = N-v+1$. Recall from Lemma \ref{lem:diffusive_kernel_expansion},
\[K^{n_1, n_2}(n_1+1-a \to n_2+1-b) = \sum_{m=0}^{n_1 \wedge N - 1}\frac{l_{n_1+1-a,n_1}\tilde{Q}_m^{(n_2)}(l_{n_2+1-b,n_2})\tilde{Q}_m^{(n_1)}(l_{n_1+1-a,n_1})}{\sqrt{(n_1+1)(n_2+2)}}\left(1 - \frac{m+1}{n_1+1}\right)^{\frac12}.\]
Applying the same argument as used in the proof Theorem \ref{thm:main_hard_edge}, one gets
\[\lim_{N\to\infty} K^{n_1, n_2}(n_1+1 - a \to n_2+1 - b) = R^{v,v-1}(a\to b). \] 
In light of \eqref{eq:diffusive_kernel_limit}, the first part of the propsition follows for $v_2 = v_1 - 1$. This same argument applied more generally to any $v_1, v_2$ also equates $R^{v_1, v_2}$ with the limit of the diffusion kernels \eqref{eq:general_diffusion_kernel_limit} thus reducing the second part of the proposition to the first. It now suffices to prove the composition identity
\begin{equation}
\label{eq:additive_polymer_composition}
R^{v_1,v_3}(a \to c) = \sum_{v_2 = 0}^{\infty}R^{v_1,v_2}(a \to b)R^{v_2,v_3}(b \to c)
\end{equation}
for all integers $v_1 \ge v_2 \ge v_3$. For this we use Lemma \ref{lem:Tit46_laguerre_expansion}. Applying the transformation of variables $y \mapsto \sqrt{1-y}$ to the equation from the lemma one gets
\[f(\sqrt{1-y}) = \sum_{a = 0}^{\infty}\tilde{J}_{b,v_2}(\sqrt{1-y}) \int_0^1\tilde{J}_{b,v_2}(\sqrt{1-z})f(\sqrt{1-z})dz.\]
Making the substitution $f(y) = \tilde{J}_{a,v_1}(y) y^{v_1-v_2}$ yields the equality
\begin{equation} 
\label{eq:Bessel_in_Bessel_Basis}
j_{a,v_1}\tilde{J}_{a,v_1}(\sqrt{1-y})(1-y)^{(v_1-v_2)/2} = \sum_{b = 0}^{\infty}j_{b,v_2}\tilde{J}_{b,v_2}(\sqrt{1-y}) R^{v_1,v_2}(a\to b).
\end{equation}
From here we can multiply both sides by $(1-y)^{(v_2-v_3)/2}$ and iterate to get
\[
j_{a,v_1}\tilde{J}_{a,v_1}(\sqrt{1-y})(1-y)^{(v_1-v_3)/2} = \sum_{c=0}^{\infty} \left(\sum_{b = 0}^{\infty}  R^{v_1,v_2}(a \to b)R^{v_2,v_3}(b \to c)\right) j_{c,v_3}\tilde{J}_{c,v_3}(\sqrt{1-y}).
\]
This expression and \eqref{eq:Bessel_in_Bessel_Basis} give us two different ways of expanding the left-hand side in terms of the Fourier-Bessel series $\tilde{J}_{c,v_3}(\sqrt{1-y})$. Equating the coefficients of these two expressions gives precisely the equality \eqref{eq:additive_polymer_composition}, as desired.
\end{proof}
We now have all the tools necessary to prove our secondary result. Once again, the overarching idea is to apply the Fourier-Bessel expansion of Lemma \ref{lem:Tit46_laguerre_expansion}.
\begin{proof}[Proof of Theorem \ref{thm:secondary_polymer_result}]
As the Gaussians $\eta_{a,v}$ are independent,
\[ 
Cov(\zeta_{a,v_1}, \zeta_{b,v_2}) = \sum_{v = -\infty}^{\min(v_1, v_2)} \sum_{s=1}^{\infty} P^{v_1,v}(a \to s)P^{v_2,v}(b \to s) \frac{j_{s,v}^2}{2}.
\]
Applying Proposition \ref{prop:additive_polymer_transition_probs}, the right-hand side equals
\[ 
\frac{j_{a,v_1}j_{b,v_2}}{2}\sum_{v = -\infty}^{\min(v_1, v_2)} \sum_{s=1}^{\infty} \int_0^1 \tilde{J}_{a,v_1}(x)\tilde{J}_{s,v}(x) (1-x)^{(v_1-v)/2}dx \int_0^1 \tilde{J}_{b,v_2}(y)\tilde{J}_{s,v}(y) (1-y)^{(v_2-v)/2}dy.
\]
Interchanging the inner summation and outer integral this becomes
\[  
\frac{j_{a,v_1}j_{b,v_2}}{2}\sum_{v = -\infty}^{\min(v_1, v_2)} \int_0^1 \tilde{J}_{a,v_1}(x)(1-x)^{(v_1-v)/2} \left(\sum_{s=1}^{\infty}\tilde{J}_{s,v}(x) \int_0^1 \tilde{J}_{b,v_2}(y)\tilde{J}_{s,v}(y) (1-y)^{(v_2-v)/2}dy\right)dx,
\]
which we simplify using \eqref{eq:Bessel_in_Bessel_Basis} to get
\[
\frac{j_{a,v_1}j_{b,v_2}}{2}\sum_{v = -\infty}^{\min(v_1, v_2)} \int_0^1 \tilde{J}_{a,v_1}(x)\tilde{J}_{b,v_2}(x) (1-x)^{(v_1+v_2-2v)/2} dx.
\]
This interchange is justified as \eqref{eq:Bessel_in_Bessel_Basis} holds in $L^2[0,1]$ and the limit of an inner product equals the inner product of the limit. Finally, as our integrands are bounded, we can interchange the $v$ summation with the integral to get the desired value.
\end{proof}

\printbibliography

\end{document}